\documentclass[12pt,reqno]{amsart}
\usepackage{a4}
\usepackage{amsmath}
\usepackage{amssymb}
\usepackage{amscd}                  
\usepackage{verbatim}
\usepackage{graphicx}               

\newcounter{theorem}
\newtheorem{cor}[theorem]{Corollary}
\newtheorem{lem}[theorem]{Lemma}

\newtheorem{pro}[theorem]{Proposition}

\newtheorem{rem}[theorem]{Remark}

\newcounter{intro}

\newtheorem{theo}[intro]{Theorem}

\newcommand{\cref}[1]{Corollary \ref{#1}}

\newcommand{\eref}[1]{(\ref{#1})}

\newcommand{\lref}[1]{Lemma \ref{#1}}

\newcommand{\subs}[1]{$\bullet$ {\it #1}}

\newcommand{\rest}{{\!}\upharpoonright} 

\newcommand{\dint}{\displaystyle \int}
\newcommand{\dlim}{\displaystyle \lim}
\newcommand{\dsum}{\displaystyle \sum}

\newcommand{\tir}{\discretionary{.}{}{---\kern.7em}}
\newcommand\I{\mathbb{I}}

\newcommand\CP{\mathbb{CP}}
\newcommand\N{\mathbb{N}}
\newcommand\R{\mathbb R}

\newcommand{\Sphere}{\mathbb S} 

  \newcommand\mcE{\mathcal E}
\newcommand\mcC{\mathcal C} \newcommand\mcI{\mathcal I}
  \newcommand\mcL{\mathcal L}
\newcommand\mcD{\mathcal D}  \newcommand\mcH{\mathcal H}

\newcommand\barM{\overline{M_1}}

\newcommand{\eps}{\varepsilon} 
\renewcommand{\phi}{\varphi}   

\newcommand{\wt}{\widetilde}           
\newcommand{\cqfd}{\hfill$\square$}

\newcommand{\Id}{\operatorname{Id}}
\newcommand{\id}{\operatorname{id}}

\renewcommand{\Im}{\operatorname{Im}}
\newcommand{\Ker}{\operatorname{Ker}}

\newcommand{\Dom}{\operatorname{Dom}}
\newcommand{\spec}{\operatorname{spec}}
\newcommand{\Spec}{\operatorname{Spec}}

\newcounter{look}
\begin{document}

\title[Partial collapsing and the spectrum]
{Partial collapsing and the spectrum of the Hodge-de Rham operator}

\author{Colette ANN\'E}
\address{Laboratoire de Math\'ematiques Jean Leray, Universit\'e de Nantes,
 CNRS, Facult\'e des Sciences, BP 92208, 44322 Nantes, France}
\email{colette.anne@univ-nantes.fr}
\author{Junya TAKAHASHI}
\address{Research Center for Pure and Applied Mathematics,
  Graduate School of Information Sciences, 
  T\^ohoku University, Aoba 6-3-09, Sendai, 980-8579, Japan}
\email{t-junya@math.is.tohoku.ac.jp}
\date{\today, 
\emph{File: }\texttt{genco-9.tex}
\\{$2010$ {\it Mathematics Subject Classification.} 
 Primary $58J50$; Secondary $35P15,$ $53C23,$ $58J32$. } 
}

\begin{abstract}
 The goal of the present paper is to calculate the limit spectrum of the 
 Hodge-de Rham operator under the perturbation of collapsing one part of 
 a manifold obtained by gluing together two manifolds with the same 
  boundary. 
 It appears to take place in the general problem of blowing-up conical 
 singularities as introduced in Mazzeo \cite{Maz} and Rowlett \cite{Row1, Row2}.  \\

 \hspace{-0.5cm}
 R\'esum\'e.\ \
 Nous calculons la limite du spectre de l'op\'erateur de Hodge-de Rham sur 
 les formes diff\'erentielles dans le cas d'\'effondrement d'une partie 
 d'une vari\'et\'e obtenue en collant deux vari\'et\'es de m\^eme bord.
 Ce r\'esultat apporte un nouvel \'eclairage aux questions de {\it blowing up 
  conical singularities} introduites par Mazzeo \cite{Maz} et Rowlett 
  \cite{Row1, Row2}. 
\end{abstract}

\maketitle



\section{Introduction.}
 
 This work takes place in the general context of the spectral studies of 
 singular perturbations of the metrics, as a manner to know what are the 
  topological or metrical meanings carried by the spectrum of 
  geometric operators. 
 We can mention in this direction, without exhaustivity, studies on the 
 adiabatic limits (\cite{MM}, \cite{R}), on collapsing (\cite{Fuk}, 
 \cite{Lot1,Lot2}), on resolution blowups of conical singularities 
 (\cite{Maz}, \cite{Row1,Row2}) and on shrinking handles (\cite{AC2, ACP}).

 The present study can be considered as a generalization of the results of
 \cite{AT}, where we studied the limit of the spectrum of the Hodge-de Rham 
  (or the Hodge-Laplace) operator under collapsing of one part of 
  a connected sum. 

  In our previous work, we restricted the submanifold $\Sigma$, 
 used to glue the two parts, to be a sphere. 
 In fact, this problem is quite related to resolution blowups of conical 
 singularities: the point is to measure the influence of the topology of 
 the part which disappears and of the conical singularity created at 
 the limit of the `big part'. 
 If we look at the situation from the `small part', we understand the 
 importance of the {\it quasi-asymptotically conical space}
 obtained from rescaling the small part and gluing an infinite cone, 
 see the definition below in (\ref{M2tilde}).

\begin{figure}[ht]\label{fig:genco-fig}
 \begin{center}
 \includegraphics*[height=6cm,width=15cm]{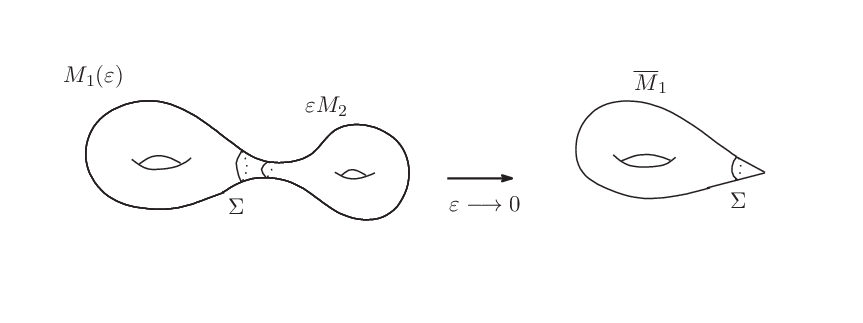}%
 \vspace{-1.5cm} 
 \caption{Partial collapsing of $M_{\varepsilon}$}
 \end{center}
\end{figure}

 When $\Sigma$ is the sphere ${\mathbb S}^n,$ 
 the conical singularity is quite simple.
 There are no {\it half-bound states}, called extended solutions 
 in the sequel, on the quasi-asymptotically conical space. 
 Our result presented here takes care of these new possibilities and 
 gives a general answer to the problem studied by Mazzeo and Rowlett. 
 Indeed, in \cite{Maz,Row1,Row2}, it is supposed that the spectrum of the 
 operator on the quasi-asymptotically conical space does not meet $0$. 
 Our study relaxes this hypothesis. 
 It is done only with the Hodge-de Rham operator, but can easily be 
 generalized.

 Let us fix some notations.

 \subsection{Set up.}

 Let $M_1$ and $M_2$ be two connected oriented compact manifolds with 
 the same boundary $\Sigma$, a compact manifold of dimension $n \geq 2$. 
 We denote by $m=n+1$ the dimension of $M_1$ and $M_2.$
 We endow $\Sigma$ with a fixed metric $h$. 

 Let $\overline{M}_1$ be the manifold with conical singularity obtained from 
 $M_1$ by gluing  $M_1$ to a cone 
   $ \mcC = [0, 1) \times \Sigma \ni(r,y)$: 
{\it there exists on $\overline{M}_1 = M_1 \cup \mcC$ a metric $\bar{g}_1$ 
 which writes, on the smooth part $r>0$ of the cone, $dr^2+r^2 h.$}

 We choose on $M_2$ a metric $g_2$ which is `trumpet like', 
 {\it i.e.\ $M_2$ is isometric near the boundary to 
 $[0, \frac{1}{2}) \times \Sigma$ with the conical metric which writes
  $ds^2+(1-s)^2 h,$ 
  if $s$ is the coordinate defining the boundary by $s=0.$}

 For any $\eps$ with $0 \le \eps < 1,$ we define 
 $$ \mcC_{\eps, 1} = \{ (r,y) \in \mcC \, | \, r > \eps \} \quad \hbox{and} 
    \quad   M_1(\eps) = M_1 \cup \mcC_{\eps, 1}.
 $$
 The goal of the following calculus is to determine the limit spectrum of 
 the Hodge-de Rham operator acting on the differential forms of the 
 Riemannian manifold
$$ M_{\eps} = M_1(\eps) \cup_{\eps.\Sigma} {\eps.} M_2
$$
 which is obtained by gluing together $(M_1(\eps), g_1)$ and $(M_2, \eps^2 g_2).$ 
 We remark that, by construction, these two manifolds have isometric 
 boundary and that the metric $g_{\eps}$ obtained on $M_{\eps}$ is smooth.
\begin{rem}
 The common boundary $\Sigma$ of dimension $n$ has some topological 
 obstructions. 
 In fact, since $\Sigma$ is the boundary of the oriented compact manifold 
$M_1$, 
  $\Sigma$ is oriented cobordant to zero.
 So, by Thom's cobordism theory, all the Stiefel-Whitney and all 
   the Pontrjagin   numbers vanish (cf.\ C.\ T.\ C.\ Wall \cite{Wall[60]}
  or \cite{Milnor-Stasheff[74]}, $\S 18$, $p.217$). 
  Furthermore, this condition is also sufficient, that is, the inverse 
 does hold. 
  Especially, it is impossible to take $\Sigma^{4k}$  
  as the complex projective spaces $\CP^{2k}, (k \ge 1),$ 
  because the Pontrjagin number $p_k (\CP^{2k}) \neq 0.$ 
\end{rem}
\subsection{Results.}

 We can describe the limit spectrum as follows: 
 it has two parts. One part comes from the big part, namely $\overline{M_1}$, 
  and is expressed by the spectrum of a good extension 
 of the Hodge-de Rham operator on this manifold with the conical singularity. 
 This extension is self-adjoint and comes from an extension of the 
 Gau\ss-Bonnet operators. 
 All these extensions are classified by subspaces $W$ of the total 
 eigenspaces corresponding to the eigenvalues within $(- \frac{1}{2}, \frac{1}{2})$ 
 of an operator $A$ acting on the boundary $\Sigma$. 
 This point is developed below in Section $\ref{GBconic}$. 
  The other part comes from the collapsing part, namely $M_2,$ 
 where the limit Gau\ss-Bonnet operator is taken with boundary conditions 
 of the Atiyah-Patodi-Singer type. 
 This point is developed below in Section \ref{GBaps}. 
 This operator, denoted $\mcD_2$ in the sequel, can also be seen on the 
 quasi-asymptotically conical space $\widetilde{M}_2$ already mentioned,
  namely 
\begin{equation}\label{M2tilde}
 \widetilde{M}_2 = M_2 \cup \Big( [1,\infty) \times \Sigma \Big)
 \end{equation}
 with the metric $dr^2+r^2 h$ on the conical part. 
 Only the zero eigenvalue is concerned with this part. 
 In fact, the manifold $M_{\eps}$ has small eigenvalues,
  in the difference with \cite{AT}, and the \emph{multiplicity of $0$} 
  at the limit corresponds to the total eigenspaces of these small and
  null eigenvalues.
 Thus, our main theorem, which asserts the convergence of the spectrum, 
 has two components.
\begin{theo}\label{A}
  The set of all positive limit values is just equal to that of all 
 positive spectrum of the Hodge-de Rham operator $\Delta_{1,W}$ 
 on $\overline{M_1},$ where
\begin{equation*}
   W \subset \bigoplus_{ |\gamma| < \frac{1}{2} } \Ker (A - \gamma)
\end{equation*}
is the space of the elements that generate {\it extended solutions}
  on $\widetilde{M_2}.$
 A precise definition is given below in $(\ref{defW})$.
\end{theo}
\begin{theo}\label{B}
  The multiplicity of $0$ in the limit spectrum is given by the sum 
\begin{equation*}
    \dim \Ker (\Delta_{1,W}) + \dim \Ker(\mcD_2) + i_{\frac{1}{2}}, 
\end{equation*}
 where $i_{\frac{1}{2}}$ denotes the dimension of the vector space 
 $\mcI_{\frac{1}{2}},$ see $(\ref{I1/2}),$ of extended solutions 
 $\omega$ on $\widetilde{M_2}$ introduced by Carron \cite{C}, 
 admitting on restriction to $r=1$
  a non-trivial component in $\Ker (A - \tfrac{1}{2}).$
\end{theo}

\subsection{Comments.}
\subsubsection{}
 Remark that this result is also valid in dimension 2. 
 In order to understand it, look at the following example. Let $I=[0,1]$
 and $M_1=M_2= \Sphere^1 \times I.$ We can shrink half of a torus : 
$\Sphere^1\times\Sphere^1=M_1\cup_\Sigma M_1$ for 
$\Sigma=\Sphere^1\sqcup\Sphere^1.$ 
 Then $\overline{M_1}$ is a 2-sphere with no harmonic 1-forms and 
 $\widetilde{M}_2$ has no $L^2$ harmonic 1-forms. 
 But $i_{\frac{1}{2}}=2.$
Indeed  $\widetilde{M}_2$ is a cylinder with flat ends. With evident
coordinates $(r,\theta),$ $d\theta$ and $\ast(d\theta)\sim\frac{dr}{r}$ 
near $\infty$ give a base for extended solutions.

\subsubsection{}We choose, in our study, a simple metric to make explicit 
 computations.
 This fact is not a restriction, as already explained in \cite{AT}, 
 because of the result of Dodziuk \cite{D} which assures uniform 
 control of the eigenvalues of geometric operators with regard to 
 variations of the metric.

\subsubsection{}
 More examples are given in the last section of the present paper.


\subsubsection*{Acknowledgement.}
 The second author was partially supported by 
 Grant-in-Aid for Young Scientists (B) $24740034$.


 \section{Gau\ss -Bonnet operator.}

 On a Riemannian manifold, the Gau\ss -Bonnet operator is defined as 
 the operator $D = d + d^{\ast}$ acting on differential forms. 
 It is symmetric and can have some closed extensions on manifolds 
 with boundary or with conical singularities. 
 We review these extensions in the cases involved in our study.
\subsection{Gau\ss -Bonnet operator on $M_{\eps}$}

 We recall that, on $M_{\eps}$, a Gau\ss -Bonnet operator $D_{\eps},$ 
 Sobolev spaces and also a Hodge-de Rham operator $\Delta_{\eps}$ 
 can be defined as a general construction on any manifold 
 $X= X_1 \cup X_2,$  which is the union of two Riemannian manifolds 
 with isometric boundaries (the details are given in  \cite{AC2}): 
 if $D_1$ and $D_2$ are the Gau\ss-Bonnet operators ``$d + d^{\ast}$'' 
 acting on the differential forms of each part, the quadratic form 
 \begin{equation}\label{defq} 
  q (\phi) = \dint_{X_1} |D_1 (\phi \rest_{X_1} ) |^2 \, d \mu_{X_1} 
      + \dint_{X_2} |D_2 (\phi  \rest_{X_2} )|^2 \, d \mu_{X_2} 
 \end{equation}
  is well-defined and closed on the domain
$$
  \Dom (q) = \{ \phi = (\phi_1,\phi_2) \in H^1(\Lambda T^{\ast} X_1) \times  
    H^1(\Lambda T^\ast X_2), \, | \,
  \phi_1 \rest_{\partial X_1} \stackrel{L^2}{=} \phi_2 \rest_{\partial X_2}
   \}.
$$
  On this space, the total Gau\ss-Bonnet operator 
  $ D(\phi) =( D_1(\phi_1), \, D_2(\phi_2)) $ 
 is defined and self-adjoint. 
 For this definition, we have, in particular, to identify 
 $(\Lambda T^\ast X_{1}) \rest_{\partial X_1}$ and
 $(\Lambda T^\ast X_{2}) \rest_{\partial X_2}.$ 
 This can be done by decomposing the forms in tangential and normal part 
 (with inner normal), the equality above means then that the tangential parts 
 are equal and the normal parts opposite. 
 This definition generalizes the definition in the smooth case.

 The Hodge-de Rham operator $(d+d^\ast)^2$ of $X$ is then defined as the 
 operator obtained by the polarization of the quadratic form $q$. 
 This gives compatibility conditions between $\phi_1$ and $\phi_2$ 
 on the common boundary. 
 We do not give details on these facts, because our manifold is smooth.
 But we shall use this presentation for the quadratic form.
 \subsection{Gau\ss -Bonnet operator on $\overline{M_1}$.}\label{GBconic}

 Let $D_{1,\min}$ be the closure of the Gau\ss -Bonnet operator defined on 
  the smooth forms with compact support in the smooth part $M_1(0).$ 
  For any such form $\phi_1,$ we write, following \cite{BS} and \cite{ACP}, on the cone $\mcC$
$$
 \phi_{1}  = dr \wedge r^{- (\frac{n}{2} -p+1)} \beta_{1,\eps} + 
              r^{- (\frac{n}{2} -p)}\alpha_{1,\eps} 
$$
 and define $\sigma_1 = (\beta_1, \alpha_1) = U(\phi_1).$
 The operator has, on the cone $\mcC$, the expression
\begin{equation*}
 U D_1 U^{\ast} =\begin{pmatrix} 0 & 1 \\
                                -1 & 0   
                \end{pmatrix}
       \Big( \partial_r + \frac{1}{r} A \Big) \;\text{ with } \; 
         A=      \begin{pmatrix} \dfrac{n}{2} -P  & - D_0 \\
                                   -D_0 & P - \dfrac{n}{2} 
                  \end{pmatrix}, 
\end{equation*}
 where $P$ is the operator of degree which multiplies by $p$ per a $p$-form, 
 and $D_0 = d_0 + d^{\ast}_0$ is the Gau\ss-Bonnet operator 
 on the manifold $(\Sigma, h)$,
 while the Hodge-de Rham operator has, in these coordinates, the expression 
\begin{equation}\label{laplace}
  U \Delta_1 U^{\ast} =  - \partial^2_r + \frac{1}{r^2} A (A+1).
\end{equation}
 The closed extensions of the operator $D_1 = d + d^{\ast}$ on 
 the manifold with conical singularity $\overline{M}_1$ has been 
 studied in \cite{BS} and \cite{L}.
 They are classified by the spectrum of its \emph{Mellin symbol}, 
 which is here the operator with parameter $A+z$. 
\subsubsection*{Spectrum of $A$}\label{spectA}\tir 
 The spectrum of $A$ was calculated in Br\"uning and Seeley \cite{BS}, 
 p.$703$.  By their result, the spectrum of $A$ is given by the values 
\begin{equation}\label{vpA}
\begin{cases}
  & \pm (p - \frac{n}{2}) \ \hbox{ with multiplicity } \dim H^p(\Sigma),  \\
  & \frac{(-1)^{p+1}}{2} \pm \sqrt{ \mu^2 + \Big( \frac{n-1}{2} - p \Big)^2 },
\end{cases}
\end{equation}
 where $p$ is any integer, $0 \leq p \leq n$ and $\mu^2$ runs over 
 the spectrum of the Hodge-de Rham operator on $(\Sigma, h)$ 
 acting on the coexact $p$-forms. 

 Indeed, looking at the Gau\ss-Bonnet operator acting on even forms, 
 they identify even forms on the cone with the sections 
 $( \phi_0, \ldots, \phi_n)$ of the total bundle $\Lambda T^{\ast} (\Sigma)$ 
 by $ \phi_0 + \phi_1 \wedge dr + \phi_2 + \phi_3 \wedge dr + \cdots.$ 
 These sections can as well represent odd forms on the cone by 
 $\phi_0 \wedge dr + \phi_1 + \phi_2 \wedge dr + \phi_3 + \cdots.$
 With these identifications, they have to study the spectrum of the 
 following operator acting on sections of $\Lambda T^\ast(\Sigma)$ 
\begin{equation*}
 S_0=
 \left( \begin{array}{ccccc}
  c_0 & d_0^\ast  & 0   & \cdots & 0 \\
  d_0 & c_1  & d_0^\ast & \ddots & \vdots \\
   0  & d_0  &  \ddots  & \ddots & 0 \\ 
  \vdots & \ddots & \ddots & c_{n-1} & d_0^{\ast} \\
   0  & \cdots & 0 & d_0 & c_n
 \end{array} \right),
\end{equation*}
 if $c_p=(-1)^{p+1}(p- \frac{n}{2}).$
 With the same identification, if we introduce the operator $\wt{S_0}$  
 having the same formula but on the diagonal the terms
 $ \wt{c_p} = (-1)^{p}(p-\frac{n}{2})=-c_p,$
 then the operator $A$ can be written as 
  $$ A = - \left( S_0 \oplus \wt{S_0} \right).$$
 The expression of the spectrum of $A$ is then a direct consequence of 
 the computations of \cite{BS}.
\subsubsection*{Closed extensions of $D_1$}\tir
 Let $D_{1, \max}$ be the the maximal closed extension of $D_1$
 with the domain 
 $$ \Dom (D_{1, \max}) = \{ \phi \in L^2 (\overline{M_1}) \, | \, 
      D_1 \phi \in L^2 ( \overline{M_1})  \,  \}.  $$
 If $\spec (A) \cap\ ( - \frac{1}{2}, \frac{1}{2} ) = \emptyset,$ 
  then $D_{1,\max} = D_{1,\min}$. 
 In particular, $D_1$ is essentially self-adjoint on the space of smooth 
 forms with compact support \emph{away} from the conical singularity.
 
 Otherwise, the quotient $\Dom(D_{1,\max})/ \Dom(D_{1,\min})$ is isomorphic to 
 \begin{equation*}
   B := \bigoplus_{ |\gamma | <  \frac{1}{2} } \Ker ( A - \gamma).
 \end{equation*}
 More precisely, by Lemma $3.2$ of \cite{BS}, there exists a surjective 
 linear map 
 \begin{equation*}
   \mcL : \Dom(D_{1,\max}) \rightarrow B
 \end{equation*}
 with $\Ker (\mcL) = \Dom( D_{1,\min} ).$ Furthermore, we have the estimate
 \begin{equation*}
   \| u(r) - r^{-A} \mcL(\phi) \|^2_{ L^2(\Sigma)} 
     \leq C(\phi) \, |r \log r| 
 \end{equation*}
 for $\phi \in \Dom(D_{1,\max})$ and $u = U (\phi)$.

 Now, for any subspace $W \subset B$, we can associate the operator 
 $D_{1,W}$ with the domain $\Dom(D_{1,W}) := \mcL^{-1}(W)$. 
 As a result of \cite{BS}, all closed extensions of $D_{1,\min}$ are 
 obtained by this way.  
 Remark that each $D_{1,W}$ defines a self-adjoint extension 
 $\Delta_{1,W}= (D_{1,W})^{\ast} \circ D_{1,W}$ of the Hodge-de Rham 
 operator, and, as a result, 
 we have $(D_{1,W})^{\ast} =D_{1,\I (W^\perp)}$, where
 \begin{equation*}
   \I = \begin{pmatrix} 
             0 & \id \\
          -\id & 0   
        \end{pmatrix}, 
    \qquad \text{i.e.} \quad 
   \I \begin{pmatrix} \beta \\ \alpha \end{pmatrix} = 
      \begin{pmatrix} \alpha \\  -\beta \end{pmatrix}.
 \end{equation*}
 This extension is associated with the quadratic form 
  $\phi \mapsto \|D \phi\|^2_{L^2} $ on the domain $\Dom (D_{1,W})$.

 Finally, we recall the results of Lesch \cite{L}. 
 The operators $ D_{1,W},$ and in particular $D_{1,\min}$ and $D_{1,\max}$, 
 are elliptic and satisfy the singular estimate (SE), see page $54$ 
 of \cite{L}, so by Proposition $1.4.6$of \cite{L} and the compactness of 
 $\overline{M_1},$ they satisfy the {\em Rellich property}:  
 the inclusion of $\Dom( D_{1,W})$ into $L^2 ( \overline{M_1} )$ is compact.

\subsection{Gau\ss -Bonnet operator on $M_2$.} \label{GBaps}
 
  We know, by the works of Carron \cite{C, C2}, following Atiyah-Patodi-Singer
  \cite{APS}, that the operator $D_2$ admits a closed extension $\mcD_2$ 
 with the domain defined by the global boundary condition 
 $$ \Pi_{\leq \frac{1}{2}} \circ U = 0,
 $$
 if $\Pi_I$ is the spectral projector of $A$ relative to the interval $I,$
 and $\leq \frac{1}{2}$ denotes the interval $( - \infty, \frac{1}{2} ].$
 Moreover, this extension is elliptic in the sense that the $H^1$-norm of 
 elements of the domain is controlled by the norm of the graph.
 Indeed this boundary condition is related to a problem on a complete 
 unbounded manifold as follows: 

 Let $\widetilde{M}_2$ denote the large manifold obtained from $M_2$ 
 by gluing a conical cylinder 
  $ \mcC_{1,\infty} = [1, \infty) \times \Sigma$ 
  with metric $dr^2+ r^2 h$ and $\widetilde{D}_2$ its Gau\ss -Bonnet 
 operator. 
 A differential form on $M_2$ admits an $L^2$-harmonic extension on 
 $\widetilde{M}_2$ precisely, when the restriction on the boundary 
 satisfies $\Pi_{\leq \frac{1}{2}} \circ U=0.$

 Indeed, from the harmonicity, these $L^2$-forms must satisfy 
   $ ( \partial_r + \frac{1}{r} A) \sigma = 0$ 
 or, if we decompose the form associated with the eigenspaces of $A$ as
 $ \sigma = \sum_{\gamma \in \Spec (A)} \sigma_{\gamma},$ 
 then the equation imposes that for all $\gamma \in \Spec(A)$ 
  there exists  $\sigma^0_{\gamma} \in \Ker (A- \gamma)$ such that 
  $ \sigma_{\gamma} = r^{-\gamma} \sigma^0_{\gamma}.$
  This expression is in $L^2 ( \mcC_{1, \infty} )$ 
 if and only if $\gamma > \frac{1}{2}$ or $\sigma^0_{\gamma} =0.$ 

 It will be convenient to introduce the $L^2$-harmonic extension operator
\begin{equation*}
\begin{matrix}
  P_2 :& \Pi_{> \frac{1}{2}}  \Big( H^{\frac{1}{2}}(\Sigma) \Big) & \to &
       L^2 (\Lambda T^\ast \mcC_{1,\infty}) \\
  & \sigma = \dsum_{\begin{subarray}{c} 
             \gamma \in \Spec (A) \\ \gamma > \frac{1}{2}
             \end{subarray} }   \sigma_{\gamma} & \mapsto 
  & P_2 (\sigma) = U^\ast \Big( 
    \dsum_{\begin{subarray}{c} 
       \gamma \in \Spec(A) \\ \gamma > \frac{1}{2}  
       \end{subarray} }
     r^{-\gamma} \sigma_{\gamma}  \Big).
\end{matrix}
\end{equation*}
 This limit problem is of the category \emph{non-parabolic at infinity} 
 in the terminology of Carron, see particularly Theorem $2.2$ of \cite{C} 
 and Proposition $5.1$ of \cite{C2}, 
 then as a consequence of Theorem $0.4$ of \cite{C}, we know that 
 the kernel of $\mcD_2$ is of finite dimension and that the graph norm of 
 the operator controls the $H^1$-norm (Theorem $2.1$ of \cite{C}).

\begin{pro}\label{D2} 
  There exists a constant $C>0$ such that for each differential form 
  $ \phi \in H^1(\Lambda T^{\ast} M_2)$ satisfying the boundary condition 
  $ \Pi_{\leq \frac{1}{2}} \circ U (\phi)=0,$  
$$
  \| \phi \|^2_{H^1(M_2)} \leq 
    C \left\{ \| \phi \|^2_{L^2(M_2)} + \| D_2 \phi \|^2_{L^2(M_2)} \right\}.
$$
 As a consequence, the kernel of $\mcD_2$, which is isomorphic to 
 $\Ker (\widetilde{D_2}),$ is of finite dimension and can be sent 
 in the total space $\sum_p H^p(M_2)$ of the absolute cohomology.
\end{pro}
  A proof of this proposition can be obtained by the same way as 
  Proposition $5$ in \cite{AT}.

\subsubsection*{Extended solutions}\tir
 Recall that Carron defined, for this type of operators, behind the
  $L^2$-solutions of $\widetilde{D_2} (\phi) =0$ which correspond to 
 the solutions of the elliptic operator of Proposition \ref{D2}, 
 {\em extended} solutions which are included in the bigger space 
 $\mathcal{W}$ defined as the closure of the space of smooth 
 $p$-forms with compact support in $\widetilde{M}_2$ 
 for the norm
\begin{equation*}
 \| \phi \|^2_{\mathcal W} := \| \phi \|^2_{ L^2(M_2) } +
             \| D_2 \phi \|^2_{ L^2( \widetilde{M}_2 )}.
\end{equation*}
 A Hardy-type inequality describes the growth at infinity of an 
 extended solution.
\begin{lem}\label{extend}
  For a function $v \in C^{\infty}_0(e, \infty)$ and a real number
   $\lambda$, we have
\begin{itemize}
 \item \ \hbox{if } $\displaystyle  \lambda \neq - \frac{1}{2}, \quad 
    ( \lambda + \frac{1}{2})^2 \int^{\infty}_e  \frac{v^2}{r^2} \, dr 
   \leq  \int^{\infty}_e \frac{1}{ r^{2 \lambda}} \big| \partial_r 
         (r^{\lambda} v) \big|^2 \, dr, $ 
 \item  \ \hbox{if } $\displaystyle  \lambda = - \frac{1}{2},  \quad  
     \frac{1}{4} \int^{\infty}_e \frac{v^2}{ r^2 | \log r|^2 } \, dr 
     \leq  \int^{\infty}_e  r | \partial_r ( r^{- \frac{1}{2}} v) |^2 \, dr. $
\end{itemize}
\end{lem}
 We remark now that, for a $p$-form $\phi$ with support in the infinite cone 
 $\mcC_{e, \infty}$, we can write
\begin{equation*}
\begin{split}
  \| D_2 \phi \|^2_{L^2 (\widetilde{M}_2) } 
  &= \sum_{\lambda \in \Spec(A)} \int^{\infty}_e
     \big\| (\partial_r + \frac{\lambda}{r} ) \sigma_{\lambda}
        \big\|^2_{L^2 (\Sigma)} \, dr \\
  &= \sum_{\lambda \in \Spec(A)} \int^{\infty}_e  \frac{1}{r^{2 \lambda} }
      \big\| \partial_r (r^\lambda \sigma_{\lambda} ) \big\|^2_{L^2 (\Sigma)}
       \, dr.
\end{split}
\end{equation*}
 Thus, as an application of Lemma \ref{extend}, we see that a kernel of 
 $\widetilde{D_2},$ which must be 
   $\sigma_{\lambda}(r) = r^{-\lambda} \sigma_{\lambda}(1)$ 
 on the infinite cone, satisfies the condition of growth at infinity of
 Lemma \ref{extend}. For $ \lambda> - \frac{1}{2}$ there is no restriction
 since $r^{-2 \lambda -2}$ is integrable near $\infty$, as well as for
 $ \lambda= - \frac{1}{2}$: if $v= r^{\frac{1}{2}} v_0$ for large $r$
 then the integral
 $\dint \dfrac{v^2}{ | r \log r |^2} \, dr$ is convergent, 
 so if we require that $\dfrac{1}{r} \phi$ is in $L^2$ then for any 
 $\lambda < - \frac{1}{2}$
$$
     \sigma_\lambda(1) =0.
$$
While the $L^2$-solutions correspond to the condition 
  $ \sigma_{\lambda}(1) = 0$ for any $\lambda \leq \frac{1}{2}.$
 As a consequence, the extended solutions which are not in $L^2$
 correspond to boundary terms with components in the total eigenspaces 
 related to the eigenvalues of $A$ in the interval $[- \frac{1}{2}, \frac{1}{2}].$
   In the case studied in \cite{AT}, there do not exist such eigenvalues 
 and we had not to take care of extended solutions. 

 More precisely, we must introduce the Dirac-Neumann operator 
 (see {\bf 2.a} in \cite{C2}) 
\begin{equation}\label{T}
\begin{split}
  T : H^{k+ \frac{1}{2}}(\Sigma) &\to H^{k- \frac{1}{2}}(\Sigma) \\
  \sigma  &\mapsto  U\circ D_2( \mcE (\sigma) ) \rest_{\Sigma},
\end{split}
\end{equation}
 where $\mcE (\sigma)$ is the solution of the Poisson problem:
\begin{equation*}
 (D_2)^2 (\mcE(\sigma)) =0  \hbox{ on } M_2  \ \text{ and } \ 
  U \circ \mcE(\sigma) \rest_{\Sigma} = \sigma \ \hbox{ on } \Sigma.
\end{equation*}
 In the same way, one can define
\begin{equation}\label{TC}
\begin{split}
  T_\mcC : H^{k+ \frac{1}{2}}(\Sigma) &\to H^{k- \frac{1}{2}}(\Sigma) \\
  \sigma  &\mapsto U\circ  D_2( \wt\mcE (\sigma) ) \rest_{\Sigma},
\end{split}
\end{equation}
 where $\wt\mcE (\sigma)$ is the solution of the Poisson problem:
\begin{equation*}
 (D_2)^2 (\wt {\mcE} (\sigma)) =0  \hbox{ on } \mcC_{1,\infty} \ \text{ and } 
  U \circ  \wt{\mcE}(\sigma) \rest_{\Sigma} = \sigma \ \hbox{ on } \Sigma.
\end{equation*}
 Then $ \Im (T_{\mcC}) = \Im (\Pi_{> \frac{1}{2}})$ is a subspace of 
  $ \Ker (T_{\mcC}) = \Im (\Pi_{\geq - \frac{1}{2}} ).$
 Carron \cite{C2} proved that this operator is continuous for $k \geq 0.$
 The $L^2$-solutions correspond to the boundary values in 
 $\Im (T) \cap \Im ( \Pi_{> \frac{1}{2}} ),$ 
 while extended solutions correspond to the space 
  $\Ker (T) \cap \Im( \Pi_{\geq - \frac{1}{2}} ).$ 
  Carron also proved that in the compact case, $\Ker (T) = \Im (T).$
 We can now define the space $W$ entering in Theorem \ref{A} :
\begin{equation}\label{defW}
\begin{split}
  W &= \bigoplus_{|\gamma|< \frac{1}{2}}  W_{\gamma},  \\ 
  \text{where } \   
  W_\gamma &= \left\{ \, \phi \in \Ker (A-\gamma) \, \big| \,
    {\exists} \eta \in \Im( \Pi_{>\gamma}) 
     \text{ s.t. } T ( \phi + \eta)=0 \, \right\}.
\end{split}
\end{equation} 

 Let us denote by 
\begin{equation}\label{I1/2} 
  \mcI_{\frac{1}{2}} := \Big( \Ker (T) \cap \Im(\Pi_{\geq \frac{1}{2}}) \Big)
       \Big/  \Big( \Ker (T) \cap \Im(\Pi_{> \frac{1}{2} }) \Big)
\end{equation}
 the space of extended solutions with non-trivial component on 
 $\Ker(A - \frac{1}{2}).$
\subsubsection*{Proof of Lemma $\ref{extend}$}
 Let $v \in C_0^{\infty} (e,\infty),$ by integration by parts and 
 the Cauchy-Schwarz inequality, we obtain, 
 for $\lambda \neq - \frac{1}{2}$
\begin{equation*}
\begin{split}
 \int^{\infty}_{e} \frac{v^2}{r^{2}} dr 
  &= \int^{\infty}_{e} \frac{1}{r^{2 \lambda+2}} |r^{\lambda} v |^2 \, dr
  = \int^{\infty}_{e}  \partial_r 
       \left\{ \frac{-1}{(2 \lambda+1 ) r^{2\lambda+1}} \right\} 
        | r^{\lambda} v |^2 \, dr \\
  &= \int^{\infty}_{e} \left\{ \frac{1}{(2 \lambda+1) r^{2 \lambda +1}}
     \right\} 2 (r^{\lambda} v) \partial_r (r^{\lambda} v) \, dr
  = \int^{\infty}_{e}  \frac{2}{(2\lambda +1)} \frac{v}{r}
     \cdot  r^{-\lambda} \partial_r (r^\lambda v)  \, dr \\
  &\leq  \frac{2}{ |2\lambda+1|} 
    \sqrt{ \int^{\infty}_{e} \frac{v^2}{r^2} \, dr} \cdot
    \sqrt{ \int^{\infty}_{e} \Big| r^{-\lambda} 
          \partial_r (r^\lambda v) \Big|^2 \, dr },
\end{split}
\end{equation*}
 which gives directly the first result of Lemma $3$.

 The second one is obtained in the same way:
\begin{equation*}
\begin{split}
  \int^{\infty}_{e} \frac{v^2}{r^2 | \log r |^2} \, dr
    &= \int^{\infty}_{e} \Big( \frac{v}{\sqrt r} \Big)^{2}
         \frac{1}{r | \log r |^2} \, dr
    = \int^{\infty}_{e} \Big( \frac{v}{\sqrt r} \Big)^{2}
       \partial_r \left( \frac{-1}{\log r} \right)  \, dr \\
   &=  \int^{\infty}_{e} \frac{2v}{\sqrt r} \partial_r 
      \left( \frac{v}{\sqrt r} \right) \cdot \frac{1}{\log r} \, dr
   =  \int^{\infty}_{e}  \frac{2v}{r\log r}  \cdot  \sqrt{r} \partial_r
         \left( \frac{v}{\sqrt r} \right)  \, dr \\
  &\leq 2 \sqrt{\int^{\infty}_{e} \frac{v^2}{r^2 | \log r |^2} \, dr}
    \cdot \sqrt{\int^{\infty}_{e} \left| \sqrt{r} \partial_r 
     \left( \frac{v}{\sqrt r} \right) \right|^2 \, dr}.
\end{split}
\end{equation*}
\cqfd

  \section{Notations and tools.}

 Let $q_{\eps}$ be the quadratic form defined on $M_\eps$
 by the formula (\ref{defq}), to write a form 
 $\phi_\eps \in \Dom (q_\eps),$ we use, as in \cite{ACP},  
 the following change of scales: with
\begin{equation*}
 \phi_{1,\eps} := {\phi}_{\eps} \rest_{M_1(\eps)}  \; \text{ and } \;
 \phi_{2,\eps} := \eps^{ \frac{m}{2} - p} {\phi_\eps} \rest_{M_2}.
\end{equation*}
 We write on the cone $\mcC_{\eps,1}$
$$
 \phi_{1,\eps} =dr \wedge r^{- (\frac{n}{2} -p+1)} \beta_{1,\eps} + 
              r^{-( \frac{n}{2} -p)} \alpha_{1,\eps} 
$$
 and define
 $\sigma_{1,\eps} = (\beta_{1,\eps}, \alpha_{1,\eps}) = U (\phi_{1,\eps}).$

 On the other part, it is more convenient to define $r= 1-s$ for 
 $s \in [0, \frac{1}{2} ]$ and write
 $ \phi_{2,\eps} =  dr \wedge r^{- (\frac{n}{2}-p+1)} \beta_{2,\eps}
     + r^{- ( \frac{n}{2} -p)} \alpha_{2,\eps}$ 
 near the boundary. Then we can define, for $r \in [\frac{1}{2}, 1]$ 
 (the boundary of $M_2$ corresponds to $r=1$)
$$
  \sigma_{2,\eps}(r) = (\beta_{2,\eps}(r), \alpha_{2,\eps}(r) ) = 
  U (\phi_{2,\eps})(r).
$$

 The $L^2$-norm, for a $p$-form on $M_1$ supported in the cone $\mcC_{\eps,1}$, 
 has the expression
\begin{equation*}
 \| \phi_\eps \|^2_{L^2(M_{\eps}) } = \int_{M_1 (\eps)} | \sigma_{1,\eps} |^2 d \mu_{g_1}
     + \int_{M_2} | \phi_{2,\eps} |^2 d \mu_{g_2}
\end{equation*}
 and the quadratic form on our study is
\begin{equation}\label{quadform}
\begin{split}
 q_{\eps} (\phi_{\eps}) &= \int_{M_\eps} |(d + d^{\ast}) \phi_\eps |^2_{g_{\eps}} \, 
            d \mu_{g_{\eps}}  \\
   &= \int_{M_1(\eps)} | UD_1U^{\ast} (\sigma_{1,\eps})|^2 \, d \mu_{g_1} + 
       \frac{1}{\eps^2} \int_{M_2} |D_2(\phi_{2,\eps})|^2 \, d \mu_{g_2}. 
\end{split}
\end{equation}
 The compatibility condition is, for the quadratic form,
 $ \eps^{\frac{1}{2}} \alpha_{1,\eps} (\eps) =\alpha_{2,\eps}(1)$ and 
 $\eps^{\frac{1}{2}} \beta_{1,\eps} (\eps) = \beta_{2,\eps}(1),$ or 
\begin{equation}\label{recol0}
  \sigma_{2,\eps}(1) = \eps^{\frac{1}{2}} \sigma_{1,\eps}(\eps).
\end{equation}
 The compatibility condition for the Hodge-de Rham operator, 
 of the first order, is obtained by expressing that
 $ D \phi_\eps \sim (UD_1U^{\ast} \sigma_{1,\eps}, {\eps}^{-1} UD_2U^{\ast}
  \sigma_{2,\eps})$  belongs to the domain of $D$. 
  In terms of $\sigma$, it gives 
\begin{equation}\label{recol1}
 \sigma^{\prime}_{2,\eps}(1) = {\eps}^{\frac{3}{2}} \sigma^{\prime}_{1,\eps}(\eps).
\end{equation}
 To understand the limit problem, we proceed to several estimates.

 \subsection{Expression of the quadratic form.}

For any $\phi$ such that the component $\phi_{1}$ is supported in the cone
  $\mcC_{\eps, 1}$, one has, with  $\sigma_1 = U (\phi_1)$ and 
  by the same calculus as in \cite{ACP}: 
\begin{align*}
  \int_{\mcC_{\eps,1} } | D_1 \phi|^2 \, d \mu_{g_\eps}
    &= \int_\eps^1 \left\| \Bigl( \partial_r + \frac{1}{r} A \Bigr) 
            \sigma_1 \right\|^2_{L^2(\Sigma)} \, dr \\
    &= \int_\eps^1  \Bigl[\, \| \sigma^{\prime}_1 \|^2_{L^2(\Sigma)} + 
      \frac{2}{r} \big( \sigma^{\prime}_1, \, A \sigma_1 \big)_{L^2(\Sigma)} 
      + \frac{1}{r^2} \| A \sigma_1 \|^2_{L^2(\Sigma)} \, \Bigr] dr.
\end{align*}

 \subsection{Limit problem}\label{limprob} 
 As a Hilbert space, we introduce
\begin{equation}\label{limitpb}
  \mcH_{\infty} := L^2 ( \overline{M}_1) \oplus \Ker (\widetilde{D_2})
         \oplus \mcI_{\frac{1}{2}}
\end{equation}
 with the space $\mcI_{\frac{1}{2}}$ defined in \eref{I1/2}, 
 and as the limit operator
  $$ \Delta_{1,W} \oplus  0 \oplus 0$$ 
 with $W$ defined in  (\ref{defW}).

 Finally, let us define

\subs{a cut-off function} $\xi_1$ on $M_1$ around the conical singularity:
\begin{equation}\label{coupure}
   \xi_1 (r) =
  \begin{cases}
    \ 1   &  \text{ if }  0 \leq  r \leq \frac{1}{2}, \\
    \ 0   &  \text{ if }  1 \leq  r,
  \end{cases}
\end{equation} 

 \subs{the prolongation operator}
\begin{align}\label{Peps}
 P_{\eps} :  H^{\frac{1}{2}}(\Sigma)  &\longrightarrow  
      H^1 (\mcC_{\eps,1} ) \\
  \sigma = \sum_{ \begin{subarray}{c} \gamma \in \Spec(A) \\ 
              \end{subarray} } \sigma_{\gamma} 
    &\mapsto   P_{\eps}(\sigma) = U^{\ast} \Big( 
   \sum_{ \begin{subarray}{c} \gamma \in \Spec(A) \end{subarray} }
    {\eps}^{\gamma- \frac{1}{2}} r^{-\gamma} \sigma_{\gamma} \Big).
     \nonumber
\end{align}
  We remark that, restricted on $\Im (\Pi_{> \frac{1}{2}})$,
  $P_{\eps}(\sigma)$ is the transplanted on $M_1 (\eps)$ 
 of $P_2(\sigma)$ (see Section 2.3), then there exists a constant $C >0$ 
 such that, for all $\sigma\in\Im (\Pi_{> \frac{1}{2}})$
\begin{equation}\label{contL2}
  \| P_2 (\sigma) \|^2_{L^2 ( \mcC_{1, \frac{1}{\eps}}) } =
  \| P_{\eps} (\sigma) \|^2_{ L^2( \mcC_{{\eps},1})}
  \leq C \sum_{\gamma > \frac{1}{2}} \| \sigma_{\gamma} \|^2_{L^2(\Sigma)} 
   = C \| \sigma \|^2_{ L^2 (\Sigma) }
\end{equation}
 and also that, if $\psi_2 \in \Dom (\mcD_2),$ then 
 $ \Big( {\xi_1} P_{\eps} (U ({\psi_2} \rest_{\Sigma}) ), \, \psi_2 \Big)$ 
  defines an element of $H^1(M_{\eps}).$

  \section{Proof of the spectral convergence.}

 We denote by $\lambda_N(\eps), N \geq 1,$ the spectrum of the total 
 Hodge-de Rham operator of $M_{\eps}$ and by $\lambda_N , N \geq 1,$ 
 the spectrum of the limit operator defined in Section \ref{limprob}. 
\subsection{Upper bound: 
  $\displaystyle \limsup_{\eps \to 0} \lambda_N(\eps) \leq \lambda_N$.}
 \label{upper}
 With the min-max formula, which says that
\begin{equation*}
 \lambda_N (\eps) = \underset{ \overset{E \subset \Dom (D_\eps) }
     {\dim E=N} }{\inf } \Big\{ \underset{\overset{\phi \in E}{
     \| \phi \| = 1 } }{ \sup } \int_{M_{\eps}} |D_{\eps} \phi|^2_{g_{\eps}}
     \,  d \mu_{g_{\eps}}  \Big\},
\end{equation*}
 we have to describe how transplant eigenforms of the limit 
 problem on $M_{\eps}$.

 We describe this transplantation term by term. 
 For the first term, we use the same ideas as in \cite{ACP}.

 For an eigenform $\phi$ of $\Delta_{1,W}$ corresponding to the eigenvalue 
 $\lambda$, $U(\phi)$ can be decomposed on an orthonormal base 
 $\{ \sigma_\gamma \}_{\gamma}$ of eigenforms of $A$ and 
 each component can be expressed by the Bessel functions. 
 For $\gamma \in (- \frac{1}{2}, \frac{1}{2}),$  it has the form
\begin{equation*}
  \Big\{ c_{\gamma} r^{\gamma +1} F_{\gamma} (\lambda r^2) + 
    d_{\gamma} r^{-\gamma} G_{\gamma} (\lambda r^2) \Big\} \sigma_{\gamma}
\end{equation*}
 where $F_{\gamma}, G_{\gamma}$ are entire functions satisfying 
 $F_{\gamma}(0) = G_{\gamma}(0)=1$ and 
 $c_{\gamma}, d_{\gamma}$ are constants.

 We remark that 
 $c_{\gamma} r^{\gamma +1} F_{\gamma} (\lambda r^2) \sigma_{\gamma}
  \in \Dom(D_{1,\min})$ 
 and also that
 $d_{\gamma} r^{-\gamma} \big( G_{\gamma}(\lambda r^2) - G_{\gamma}(0) \big)
  \sigma_{\gamma} \in \Dom(D_{1,\min})$.
 So we can write $\phi = \phi_0 + \overline{\phi}$ with
\begin{gather*}   
  \phi_0  \in \Dom(D_{1,\min})  \quad \hbox{and} \\
  U ( \overline{\phi} ) (r) = \xi_1(r)  
  \dsum_{\begin{subarray}{c} \gamma \in \Spec(A) \\ |\gamma|< \frac{1}{2}
         \end{subarray}} 
      d_{\gamma}r^{-\gamma} \sigma_{\gamma}.
\end{gather*}
 By the definition of $D_{1,\min},$  $\phi_0 $ can be approached, 
  with the operator norm, by a sequence of smooth forms $\phi_{0,\eps}$ with 
compact  support in $M_1(\eps).$

 By the definition of $W$, we know that 
 $\dsum_{|\gamma|<\frac{1}{2}} d_{\gamma} \sigma_{\gamma}\in W.$
 So there exists 
 $\phi_{2, \gamma} \in \Ker (D_2)$ such that
 $U ( \phi_{2,\gamma}(1) ) - d_{\gamma} \sigma_\gamma 
   \in \Im (\Pi_{>\gamma}).$
 We remark finally that, by the definition (\ref{Peps}), we can write
 $ U (\overline{\phi}) (r) = \xi_1(r) \dsum_{|\gamma|<\frac{1}{2}}
    {\eps}^{ \frac{1}{2} -\gamma}  P_{\eps} (d_{\gamma} \sigma_{\gamma}).$

 Let $\phi_{2,\eps} = \dsum_{|\gamma|< \frac{1}{2}} 
    {\eps}^{ \frac{1}{2} - \gamma} \phi_{2,\gamma}$
 and
\begin{equation*}
\begin{split}
 \phi_{\eps} &= \Big( \phi_{0,\eps} + \xi_1 P_{\eps} \Big( 
    \dsum_{\begin{subarray}{c} \gamma \in \Spec(A) \\ 
      |\gamma|<\frac{1}{2}\end{subarray}}
    {\eps}^{\frac{1}{2}-\gamma} U (\phi_{2,\gamma}(1))  \Big), 
    \phi_{2,\eps} \Big) \in H^1(M_\eps).
\end{split}
\end{equation*}
 It is a good transplantation:
  $\|\phi_{2,\eps}\|\to 0$ as the term added on $M_1 (\eps)$ 
(indeed, a term of the sum 
 $\xi_1 {\eps}^{\frac{1}{2}-\gamma} P_{\eps} 
   (U \phi_{2,\gamma}(1) -  d_{\gamma} \sigma_{\gamma})$ 
 corresponds to some $\gamma^{\prime} >\gamma,$ 
 if $\gamma^{\prime} > \frac{1}{2}$, by (\ref{contL2}),  
   it is  $O ({\eps}^{ \frac{1}{2}-\gamma})$, 
 if $\gamma^{\prime} <  \frac{1}{2}$, 
   it is  $O(\eps^{\gamma^{\prime} -\gamma})$ 
 and if $\gamma^{\prime} = \frac{1}{2}$, it is 
  $O( \eps^{\frac{1}{2}- \gamma} \sqrt{ |\log \eps|} )$).
 Moreover they are harmonic, up to $\xi_1$.

 For the two last ones, we shrink the infinite cone on $M_1$ and 
 cut with the function $\xi_1,$ already defined in (\ref{coupure}).

 Finally, if $\Ker( A- \frac{1}{2}) \neq \{ 0 \},$  
 for each non-zero element 
  $[ \overline{\sigma}^{\frac{1}{2}} ] \in \mcI_{ \frac{1}{2}}, $
 there exists  $\psi_2$ with $ D_2 (\psi_2) = 0$ on $M_2$ and 
 the boundary value  $\overline{\sigma}^{\frac{1}{2}}$ modulo 
 $\Im (\Pi_{> \frac{1}{2}} ).$ 
 Then, we can construct a {\em quasi-mode} as follows:
\begin{equation}\label{1/2}
  \psi_{\eps} := | \log {\eps}|^{- \frac{1}{2}} \Big( {\xi_1}. 
    \big\{ r^{- \frac{1}{2}} U^{\ast}  (\overline{\sigma}^{\frac{1}{2}}) + 
    P_{\eps} ( U ( \psi_2) \rest_{\Sigma} 
     - \overline{\sigma}^{\frac{1}{2}} ) \big\}, \, \psi_2 \Big)
\end{equation}
 The $L^2$-norm of this element is uniformly bounded from above 
 and below, and 
\begin{equation*}
  \dlim_{ {\eps}\to 0} \| \psi_{\eps} \|_{L^2(M_{\eps})}
    = \| \overline{\sigma}^{\frac{1}{2}} \|_{L^2(\Sigma)}.
\end{equation*}
 Moreover, it satisfies 
  $q (\psi_{\eps}) = O( | \log {\eps} |^{-1})$ 
  giving then a `small eigenvalue',
 as well as the elements of $\Ker (\mcD_2)$ and of $\Ker (\Delta_{1,W})$.

 [{\it n.b.} It is remarkable that the same construction, 
 for an extended solution with corresponding boundary value in 
  $\Ker (A- \gamma), \, \gamma \in (- \frac{1}{2}, \frac{1}{2})$ 
 does not give a quasi-mode: 
 indeed if $\psi_2$ is such a solution, the transplanted element will be 
\begin{equation*}
\begin{split}
 \psi_{\eps} &= \Big( {\xi_1}. \big\{ r^{-\gamma} U^{\ast} 
   (\overline{\sigma}^{\gamma} ) + {\eps}^{\frac{1}{2}- \gamma} P_{\eps}
    ( U (\psi_2) \rest_{\Sigma} - \overline{\sigma}^{\gamma} ) \big\}, \ 
   {\eps}^{\frac{1}{2} - \gamma} \psi_2 \Big), 
\end{split}
\end{equation*}
 for which $q (\psi_{\eps})$ does not converge to $0$ as $\eps \to 0.$ ]

 To conclude the estimate of upper bounds, we have only to verify 
 that these transplanted forms have a Rayleigh-Ritz quotient comparable to 
 the initial one and that the orthogonality is fast conserved by transplantation.

\subsection{Lower bound : 
 $\displaystyle \liminf_{\eps \to 0} \lambda_N(\eps) \geq \lambda_N$.}
 We first proceed for one index. 
 We know, by Section \ref{upper}, that for each $N,$ the family 
 $\{ \lambda_N (\eps) \}_{\eps >0}$ is bounded, set
$$
  \lambda := \liminf_{\eps \to 0} \lambda_N (\eps).
$$
 There exists a sequence $\{ \eps_i \}_{i \in \N }$ such that 
 $ \dlim_{i \to \infty} \lambda_N (\eps_i) = \lambda. $
 Let, for each $i$, $\phi_{i}$ be a normalized eigenform 
 relative to
  $\lambda_i = \lambda_N (\eps_i).$

\subsubsection{On the regular part of $\overline{M_1}$.}

\begin{lem}\label{l0}
 For our given family $\phi_{i}$, the family 
 $\{ (1-\xi_1). \phi_{1,i} \}_{i \in \N }$ is bounded in 
 $H^1_0(M_1(0), g_1).$
\end{lem}
 Then it remains to study $\xi_1. \phi_{1,i}$ which can be expressed 
 with the polar coordinates. 
 We remark that the quadratic form of these forms is uniformly bounded.


\subsubsection{Estimates of the boundary term.} 

 The expression above can be decomposed with respect to the eigenspaces 
 of $A$; in the following calculus, we suppose that $\sigma_1(1)=0$:
\begin{equation*}
\begin{split} 
  \int_\eps^1 \Bigl[\, & \| \sigma^{\prime}_1 \|^2_{L^2(\Sigma)} +
    \frac{2}{r} (\sigma^{\prime}_1, \, A \sigma_1)_{L^2(\Sigma)}
           + \frac{1}{r^2} \| A \sigma_1 \|^2_{L^2(\Sigma)} \,\Bigr] \, dr \\
   &= \int_\eps^1 \Bigl[\, \| \sigma^{\prime}_1 \|^2_{L^2(\Sigma)} + \partial_r 
       \Bigl( \frac{1}{r} (\sigma_1, \, A \sigma_1)_{L^2(\Sigma)} \Bigr)
       + \frac{1}{r^2} \Big\{ (\sigma_1, \, A \sigma_1)_{L^2(\Sigma)} 
       + \| A \sigma_1 \|^2_{L^2(\Sigma)} \Big\} \,\Bigr]  \, dr \\
   &= \int_{\eps}^1 \Bigl[\, \| \sigma^{\prime}_1 \|^2_{L^2(\Sigma)} +
       \frac{1}{r^2} ( \sigma_1, \ (A + A^2) \sigma_1)_{L^2(\Sigma)}
       \,\Bigr] \, dr - \frac{1}{\eps} 
        ( \sigma_1(\eps), A \sigma_1(\eps) )_{L^2(\Sigma)}.
\end{split}
\end{equation*}
 This shows that the quadratic form controls the boundary term, 
 if the operator $A$ is negative but $(A + A^2)$ is  non-negative. 
 The latter condition is satisfied exactly on the orthogonal complement 
 of the spectral space corresponding to the interval $(-1, 0)$. 
 By applying $\xi_1.\phi_{1,i}$ to this fact,  
 we obtain the following lemma:
\begin{lem}\label{l1}
 Let $\Pi_{\leq -1}$ be the spectral projection of the operator $A$ 
  relative to the interval $(-\infty, -1]$. 
 There exists a constant $C>0$ such that, for any $i \in \N$
$$
 \| \Pi_{\leq-1} \circ U (\phi_{1,i}(\eps_i)) 
   \|_{H^{\frac{1}{2}} (\Sigma)} \leq C \sqrt{\eps_i}.
$$
\end{lem}
 In view of Proposition \ref{D2}, we want also a control of the components 
 of $\sigma_1$ associated with the eigenvalues of $A$ in $(-1, \frac{1}{2} ]$. 
  The number of these components is finite and we can work term by term.
 So we write, on $\mcC_{\eps,1},$
$$ \sigma_1(r) = \sum_{\gamma \in \Spec(A)} {\sigma_1}^{\gamma}(r) \ 
   \text{ with } \  A {\sigma_1}^{\gamma} (r) = \gamma {\sigma_1}^{\gamma} (r)
$$
 and we suppose again $\sigma_1(1)=0.$  From the equation 
 $ ( \partial_r + \frac{A}{r} ) \sigma^{\gamma}_1 
  = r^{- \gamma} \partial_r ( r^{\gamma} \sigma^{\gamma}_1 )$ and 
  the Cauchy-Schwarz inequality, it follows that 
\begin{equation*}
\begin{split}
 \left\| {\eps}^{\gamma} \sigma^{\gamma}_1(\eps) \right\|^2_{L^2 (\Sigma)}
   &= \Big \|  \int_{\eps}^1 \partial_r (r^\gamma \sigma_1^{\gamma}) 
         \, dr \Big\|_{L^2 (\Sigma)}^2  \\
   &\le \left\{  \int_{\eps}^1 \Big\| r^{\gamma} \cdot ( \partial_r +
        \frac{1}{r} A  ) \sigma_1^{\gamma}(r) \Big\|_{L^2 (\Sigma)} \, 
            dr \right\}^2  \\
   &\leq \int_{\eps}^1 r^{2 \gamma} dr \cdot \int_{\eps_i}^1 
       \Big\| \partial_r (\sigma_1^\gamma) + \frac{\gamma}{r}
            (\sigma_1^{\gamma} ) \Big\|^2_{L^2 (\Sigma)} dr. 
\end{split}
\end{equation*}
 Thus, if the quadratic form is bounded, there exists a constant $C>0$ 
 such that
\begin{equation}
 \| \sigma_{1}^{\gamma} (\eps) \|^2_{L^2 (\Sigma)} \leq 
 \begin{cases}
   C {\eps}^{-2 \gamma} \dfrac{1- {\eps}^{2\gamma+1} }{ 2\gamma +1}
       & \hbox{ if }   \gamma \neq - \frac{1}{2}, \\
   C {\eps} |\log \eps| & \hbox{ if } \gamma= - \frac{1}{2}.
\end{cases}
\end{equation}
 This gives
\begin{lem}\label{l2}
 Let $\Pi_{I}$ be the spectral projection of the operator $A$ 
 relative to the interval $I$. There exist constants $\alpha, C>0$ 
 such that, for any $i \in \N$
$$
 \| \Pi_{(-1,0)} \circ U( \phi_{1,i}(\eps_i) ) 
   \|_{H^{\frac{1}{2}} (\Sigma)} \leq  C \eps_i^{\alpha}.
$$
\end{lem}
 Here, $0 < \alpha < \frac{1}{2}$ satisfies that $- \alpha$ is larger 
 than any negative eigenvalue of $A$. 

 With the compatibility condition (\ref{recol0}) 
 and the ellipticity of $A$, the estimate above gives also
\begin{lem}\label{l3}
 With the same notation, there exist constants $\beta, C >0$ such that, 
 for any $i\in\N$ 
$$
  \| \Pi_{[0, \frac{1}{2})} \circ  U(\phi_{2,i}(1)) 
  \|_{H^{\frac{1}{2}} (\Sigma)}  \leq C {\eps_i}^\beta.
$$
\end{lem} 
 Here, $\frac{1}{2} - \beta$ is the largest non-negative eigenvalue of 
 $A$ strictly smaller than $\frac{1}{2}$ (if there is no such a eigenvalue, 
 we put $\beta = \frac{1}{2}$).

 Finally, we study $\sigma_1^{\frac{1}{2}}$ for our family of forms
  (the parameter $i$ is omitted in the notation). 
 It satisfies, for $\eps_i < r < \frac{1}{2},$ the equation 
$$
  \Big( - \partial_r^2 + \frac{3}{4 r^2} \Big) \sigma_1^{\frac{1}{2}}
     = \lambda_i \sigma_1^{\frac{1}{2}}.
$$
 The solutions of this equation have expression in terms of the Bessel 
 and the Neumann functions: there exist entire functions $F, G$ with 
 $F(0)=G(0)=1$ and differential forms $c_{i}, \, d_{i}$ 
 in $\Ker (A- \frac{1}{2})$ such that
\begin{equation}\label{sigma1/2}
\begin{split}
  \sigma_1^{\frac{1}{2}}(r) &= c_{i} r^{\frac{3}{2}} 
    F (\lambda_{i} r^2) + d_{i} \Big\{ r^{- \frac{1}{2}} 
    G( \lambda_{i} r^2) + \frac{2}{\pi} \log(r) 
    r^{ \frac{3}{2} } F(\lambda_i r^2) \Big\}
\end{split}
\end{equation}
 (cf.\ \cite{ACP}, Lemma $4$).
 The fact that the $L^2$-norm is bounded gives that 
  $\| c_{i} \|^2_{L^2} + |\log \eps_i| \, \| d_i \|^2_{L^2}$ 
 is bounded. 
 Finally, by reporting this estimate in the expression above, 
 we have 
$$
  \| \sigma_1^{\frac{1}{2}}(\eps_i) \|^2_{L^2 (\Sigma)}= 
     O \Big( \dfrac{1}{ \eps_i |\log \eps_i| } \Big).
$$
 With the compatibility condition (\ref{recol0}), we obtain 
\begin{lem}\label{l3b}
 There exists a constant $C>0$ such that, for any $i\in\N$
$$
 \| \Pi_{ \{ \frac{1}{2} \} } \circ U( \phi_{2,i})(1) 
    \|_{H^{\frac{1}{2}} (\Sigma)} \leq  \frac{C}{ \sqrt{ |\log \eps_i|} }.
$$
\end{lem} 

\subsubsection{Convergence of $\phi_{2,i}$} \label{onM2}

 Let us now define, in general, $\wt{\phi}_{2,\eps}$ as the form 
obtained by the prolongation of $\phi_{2,\eps}$ by 
  $\sqrt{\eps} \xi_1(\eps r) \phi_{1,\eps}(\eps r)$
 on the infinite cone  $\mcC_{1,\infty}.$
 A change of variables gives that
\begin{equation*}
 \|\wt{\phi}_{2,\eps}\|_{L^2 (\mcC_{1,\infty})}
     = \|\xi_1\phi_{1,\eps}\|_{L^2 (\mcC_{\eps,1})},
\end{equation*}
 while
\begin{equation*}
\begin{split}
 \int_{\wt{M}_2} |\wt{D}_2(\wt{\phi}_{2,\eps})|^2 d \mu &= 
  \eps^2 \int_{\mcC_{\eps,1}} |D_1(\xi_1\phi_{1,\eps}) |^2 d \mu_{g_1}
  +  \int_{M_2} |D_2(\phi_{2,\eps})|^2 d \mu_{g_2}.
\end{split}
\end{equation*}
 Thus, by the definition of $\phi_i$, the family 
 $\{ \wt{\phi}_{2,i} \}_{i \in \N}$ is bounded in the space $\mathcal{W}$ 
 and
   $\dint_{\mcC_{1, \infty}} |\wt{D}_2 (\wt{\phi}_{2,i})|^2 d \mu = O(\eps_i^2).$ 
 The works of Carron \cite{C} give us that 
  $\|\wt{\phi}_{2,i}(1)\|_{H^{\frac{1}{2}}(\Sigma)}$ 
  is bounded and the following
\begin{pro}\label{inkerD2}
 There exists a subfamily of the family $\{ \wt{\phi}_{2,i} \}_{i \in \N}$ 
 which converges in $L^2(M_2, g_2)$. 
 Its limit $\wt{\phi}_2$ defines an extended solution on $\wt{M}_2$,
  i.e.\ $\wt{D}_2 (\wt{\phi}_2) =0$ and 
  $\wt{\phi}_2 \rest_{\Sigma} \in \Ker (T) \cap 
     \Im( \Pi_{\geq - \frac{1}{2}} ).$
\end{pro}
 We still denote by $\wt{\phi}_{2,i}$ the subfamily obtained. 

 \subsubsection{Convergence near the singularity.} 

 Now we use the fact that eigenforms satisfy an equation which imposes 
 a local form. We concentrate on $\gamma \in [- \frac{1}{2}, \frac{1}{2}].$ 
 If we write 
\begin{equation*}
  \phi^{[- \frac{1}{2}, \frac{1}{2}]}_{1,i} 
   = \dsum_{\gamma \in [- \frac{1}{2}, \frac{1}{2}]}
         U^{\ast} \sigma_1^\gamma(r),
\end{equation*}
 the terms $\sigma_1^\gamma$ satisfy the equations
$$
  \Big( - \partial_r^2 + \frac{\gamma(1+\gamma)}{ r^2} \Big) \sigma_1^{\gamma}
     = \lambda_{i} \sigma_1^{\gamma}.
$$
 The solutions of this equation have expression in term of the Bessel 
 functions: there exist entire functions $F, G$ with $F(0)=G(0)=1$ and 
 differential forms $c_{\gamma,i}, \, d_{\gamma,i}$ in $\Ker (A- \gamma)$ 
 such that
\begin{equation}\label{sigma-gamma}
  \sigma_1^{\gamma}(r) = 
\begin{cases}
  c_{\gamma,i} r^{\gamma +1} F_{\gamma} (\lambda_i r^2 ) + 
  d_{\gamma,i} \Big( r^{-\gamma} G_{\gamma}( \lambda_i r^2) \Big) \
   & ( |\gamma|< \frac{1}{2}),  \\
  c_{\frac{1}{2}, i} r^{\frac{3}{2}}  F_{\frac{1}{2}} (\lambda_i r^2) + 
  d_{\frac{1}{2}, i} \Big( r^{- \frac{1}{2}} G_{\frac{1}{2}}( \lambda_i r^2) 
  + \frac{2}{\pi} \log(r) r^{ \frac{3}{2} } F_{\frac{1}{2}}(\lambda_i r^2) \Big)
   &  (\gamma = \frac{1}{2}),  \\
  c_{- \frac{1}{2}, i} r^{\frac{1}{2}}  F_{-\frac{1}{2}} (\lambda_i r^2)
    + d_{- \frac{1}{2}, i} \Big( r^{\frac{1}{2}} \log(r) 
    G_{- \frac{1}{2}}( \lambda_i r^2) \Big)  
       &  (\gamma = - \frac{1}{2}). 
\end{cases}
\end{equation}
 The lemmas of the previous section give us the result that 
 the families $c_{\gamma, i}$ and $d_{\gamma, i}$ are bounded and 
 by extraction we can suppose that they converge. 
 In the case of $\gamma= \frac{1}{2}$, we have more: 
 $ \| d_{\frac{1}{2}, i} \|_{L^2 (\Sigma)} 
     = O( |\log \eps_i|^{- \frac{1}{2}}). $

 But we know also, turning back to the family of the last proposition, 
  that the family
 $ \sqrt{\eps_i}\xi_1(\eps_i r) \phi_{1,i}(\eps_i r)$ converges 
 on any sector $1\leq r \leq R$ to $0$, 
  according to the explicit form of $\sigma_1^{\gamma}(r).$ 
{\it As a consequence, the form $\wt{\phi}_2$ has no component for
 $\gamma\in [- \frac{1}{2}, \frac{1}{2}]$ and is indeed an $L^2$-solution.} 
 We have proved
\begin{pro}\label{phi2}
 The form $\wt{\phi}_2$ in Proposition \ref{inkerD2} has no component
  for $\gamma \in [- \frac{1}{2}, \frac{1}{2}].$
   If we set $\phi_2 := \wt{\phi}_2 \rest_{M_2},$
 there exists a subfamily of $\{ \phi_{2,i} \}_i$ which 
  converges, as $i \to \infty,$ to $\phi_2$ and it satisfies 
$$ \phi_2 \in \Dom (\mcD_2), 
    \  \| \phi_2 \|_{L^2(M_2, g_2)} \leq 1 \hbox{ and }  D_2(\phi_2) =0.
$$
\end{pro}

 Moreover, the harmonic prolongation of 
 $\sqrt{\eps_i} \xi_1 (\eps_i r) \phi_{1,i} (\eps_i r)$ 
 $$ \overline{\phi}_{2,i} = \mcE (\sqrt{\eps_i} \xi_1
   (\eps_i r) \phi_{1,i} (\eps_i r)) $$
 minimizes the norm of $D_2(\phi_2)$. As a consequence,
 $ \| D_2 (\overline{\phi}_{2,i})\|_{L^2(M_2)}= O (\eps_i)$ 
 implies
\begin{equation*}
 \| T (\sqrt{\eps_i} \phi_{1,i} (\eps_i)) \|_{H^{- \frac{1}{2}}
     (\Sigma)} = O(\eps_i)
\end{equation*}
 with the Dirac-Neumann operator $T$ defined in (\ref{T}). 

 But, by Lemmas \ref{l1} and \ref{l2}, we know that
 $\| \Pi_{< - \frac{1}{2}} (\phi_{1,i}(\eps))
    \|_{H^{\frac{1}{2}}(\Sigma)} = O( \sqrt{\eps}).$
  The continuity of $T$ gives hence
 $ \| T \circ \Pi_{\geq - \frac{1}{2}} (\phi_{1,i}(\eps_i)) 
  \|_{H^{- \frac{1}{2}}(\Sigma)}  =O (\sqrt{\eps_i}). $
 To obtain consequences of this result on the term
 $\Pi_{[-\frac{1}{2}, \frac{1}{2}]}(\phi_{1,i} (\eps_i)),$ 
 we must make sense of the possibility of working modulo $\Im (T)$. 
 In the following, for simplicity of notation, we identify 
 the spectral projection $\Pi_I$ of $A$ for the interval $I$ 
 with $U^\ast \Pi_I U.$ 
\begin{pro}\label{Bphi}
 The space $T( \Im (\Pi_{> \frac{1}{2}}) \cap H^{\frac{1}{2}}(\Sigma))$
  is closed in $H^{- \frac{1}{2}}(\Sigma)$, 
   as a consequence of the works of Carron.
 Let us define $B(\phi)$ for 
 $\phi \in \Im (\Pi_{[- \frac{1}{2}, \frac{1}{2}]})$ 
 as the orthogonal projection of $T(\phi)$ onto the orthogonal complement 
 of this space 
  $T (\Im (\Pi_{> \frac{1}{2}}) \cap H^{\frac{1}{2}}(\Sigma)).$
 Then $B$ is linear and satisfies 
\begin{equation*}
\begin{split}
 \bullet  & \quad  \| B \phi \|_{H^{-\frac{1}{2}} (\Sigma)}
         \leq  \| T \phi \|_{H^{- \frac{1}{2}} (\Sigma)},   \\
 \bullet  & \quad   \text{If } B (\phi) =0,  \text{ there exists an } 
     \eta \in \Im (\Pi_{> \frac{1}{2}}) \ \text{such that} 
    \ T ( \phi + \eta) =0.
\end{split}
\end{equation*}
\end{pro}
\begin{proof} 
 To prove that $T (\Im (\Pi_{>\frac{1}{2}}) \cap H^{\frac{1}{2}}(\Sigma))$ 
 is closed, we must recall some facts contained in \cite{C2}. 
 Let us denote here $T_{\mcC}$ the operator constructed as $T$, 
 but for the infinite part $\mcC_{1,\infty}.$ 
 Then $ \Im (T_{\mcC}) = \Im (\Pi_{> \frac{1}{2}})$ is a subspace of
  $ \Ker (T_\mcC) = \Im (\Pi_{\geq - \frac{1}{2}} ).$
 We know that $T+ T_{\mcC}$ is an elliptic operator of order $1$ on $\Sigma$ 
 which is compact. 
 As a consequence, $\Ker(T+T_\mcC)$ is finite dimensional, 
 $(T+T_\mcC) (H^{\frac{1}{2}}(\Sigma))$ is a closed subspace 
 of $H^{- \frac{1}{2}}(\Sigma)$ and $T+T_\mcC$ admits a continuous
  parametrix $Q : H^{-\frac{1}{2}}(\Sigma) \to H^{\frac{1}{2}}(\Sigma)$
  such that
\begin{equation*}
  Q \circ (T + T_{\mcC}) = \Id - \Pi_{\Ker (T+T_{\mcC})},
\end{equation*}
 where $\Pi_{\Ker(T+T_\mcC)}$ denotes the orthogonal projection
 onto $\Ker(T+T_\mcC)$ for the inner product of $H^{\frac{1}{2}}(\Sigma).$ 
 We can know prove that
 $T (\Im \Pi_{> \frac{1}{2}} \cap H^{\frac{1}{2}}(\Sigma))$ is closed. 

 Let $\{ \sigma_i \}_i$ be a sequence of elements in 
  $\Im (\Pi_{> \frac{1}{2}} ) \cap H^{\frac{1}{2}}(\Sigma)$
 such that $T (\sigma_i)$ converges, and 
  let $\psi = \lim_{i \to \infty} T( \sigma_i).$
 We can suppose that 
\begin{equation*}
 \sigma_i \in \Big( \Ker (T) \cap \Im (\Pi_{> \frac{1}{2}}) 
     \cap H^{\frac{1}{2}} (\Sigma) \Big)^{\perp}.
\end{equation*}
 We have $\Im (\Pi_{> \frac{1}{2}}) \cap H^{\frac{1}{2}}(\Sigma)
   \subset \Ker (T_{\mcC}).$ 
 Then it means that
 $ (T+T_{\mcC}) \sigma_i = T ( \sigma_i)$ converges and 
 $\tau_i = Q \circ (T+ T_{\mcC}) \sigma_i$ 
 converges, let $\tau = \lim_{i \to \infty} \tau_i.$
 Thus,
\begin{equation*}
  \sigma_i = \tau_i + e_i \ \text{ with } \
     \tau_i \in  \Ker (T+ T_{\mcC})^{\perp}, \
      e_i \in \Ker(T+T_{\mcC} ).
\end{equation*}

 The sequence $\{ e_i \}_i$ must be bounded, unless we can extract a 
 subsequence $\| e_i \|\to \infty$, so it is true also for $\| \sigma_i \|$
 and by extraction we can suppose that the bounded sequence
  $e_i / \| \sigma_i \|$  converges,
 since it leaves in a finite dimensional space.
 Let $e^{\prime}$ be this limit, then $e^{\prime} = \lim e_i / \| \sigma_i \|$ 
 also and 
 $e^{\prime} \in \Im (\Pi_{> \frac{1}{2}}) \cap H^{ \frac{1}{2}}(\Sigma).$ 

 Finally, $e^{\prime}$ satisfies $\| e^{\prime} \|= 1$ and
\begin{equation*}
  e^{\prime} \in \Ker( T + T_{\mcC})  \hbox{ and } \, 
   e^{\prime} \in \Ker (T_{\mcC}),   
\end{equation*}
 as well as $e_i$ and $\sigma_i$, which implies $T (e^{\prime}) = 0.$
 Thus, $e^{\prime} = \lim \sigma_i / \| \sigma_i \| \in 
   \Im (\Pi_{> \frac{1}{2}} ) \cap H^{\frac{1}{2}}(\Sigma) \cap \Ker (T).$ 
 But, by the assumption of $\sigma_i,$ 
 $e^{\prime}$ must be orthogonal to this space, which is a contradiction.

 So, $e_i$ is a bounded sequence in a finite dimensional space, 
 by extraction, we can suppose that it converges.
 Then $\sigma_i$ admits a convergent subsequence, 
 and let $\sigma$ denote its limit:
\begin{equation*}
 \sigma \in \Im (\Pi_{> \frac{1}{2}} ) \cap H^{\frac{1}{2}}(\Sigma) \ 
   \hbox{ and } \psi = T (\sigma).
\end{equation*}
\end{proof}

 As an application of this Proposition \ref{Bphi}, we have
\begin{equation*}
 \| B \circ \Pi_{[-\frac{1}{2}, \frac{1}{2}]} (\phi_{1,i}(\eps_i)) 
       \|_{H^{- \frac{1}{2}} (\Sigma)} = O (\sqrt{\eps_i}).
\end{equation*}
 This is the sum of few terms.
 We remark that the term with $c_{\gamma, i}$ is in fact always 
 $O (\sqrt{\eps_i}).$ 
 For the same reason, we can freeze the function $G$ at $0$,
  where its value is $1$. So we can say 
\begin{equation}\label{e18}
 \Big\| \eps_i^{\frac{1}{2}} \log (\eps_i) B \circ U^{\ast} 
  (d_{- \frac{1}{2},i}) + \sum_{|\gamma|< \frac{1}{2}} \eps_i^{-\gamma} 
  B \circ U^{\ast} (d_{\gamma,i}) + 
    \eps_i^{- \frac{1}{2}} B \circ U^{\ast} (d_{\frac{1}{2},i})
  \Big\|_{H^{- \frac{1}{2}} (\Sigma)}
    =  O( \sqrt{\eps_i}), 
\end{equation}
 while all the other terms, which have the behavior of 
 $r^{\delta}$ with $\delta > \frac{1}{2},$ enter in an expression 
 belonging to $\Dom(D_{1,\min}).$

 In fact, we have the following result.
\begin{pro}\label{phi1W}
 One can write 
 $\Pi_{(- \frac{1}{2}, \frac{1}{2}]} \circ U (\xi_1 \phi_{1,i}) =
    \overline{\sigma}_{1,i} + \sigma_{0,i}$ with the bounded sequence 
  $U^{\ast} (\sigma_{0,i}) \in \Dom(D_{1,\min})$ and 
  $\overline{\sigma}_{1,i} = \overline{\sigma}_{1,i}^{< \frac{1}{2}} +
    \overline{\sigma}_{1,i}^{\frac{1}{2}}$ 
 satisfies that there exists a subfamily of 
  $\overline{\sigma}_{1,i}^{<\frac{1}{2}}$ 
  which converges, as $i \to \infty$ to 
 $\dsum_{\gamma \in (- \frac{1}{2}, \frac{1}{2})} r^{-\gamma} \sigma_\gamma $ 
  with
 $\dsum_{\gamma \in (- \frac{1}{2}, \frac{1}{2})}\sigma_\gamma \in W, $
 while 
\begin{equation*}
 \overline{\sigma}_{1,i}^{\frac{1}{2}} \sim 
   \frac{1}{\sqrt{|\log \eps_i|}}
     r^{- \frac{1}{2}} \overline{\sigma}_{\frac{1}{2}}  \
   \text{ for some } \  \overline{\sigma}_{\frac{1}{2}} 
   \in \Ker(A- \tfrac{1}{2}).
\end{equation*}
 Thus, $\overline{\sigma}_{1,i}^{\frac{1}{2}}$ concentrates 
 on the singularity .
\end{pro}

\begin{proof}
The term $\overline{\sigma}_{1,i}$ comes from the expression obtained
in \eref{e18}, while $\sigma_{0,i}$ is the sum of all the other terms.

We then concentrate on \eref{e18}.
First, we gather the terms concerning the same eigenvalue and 
still denote by $d_{\gamma,i}$ the sum of all the terms with the same 
eigenvalue. 
Let $- \frac{1}{2} \leq \gamma_p < \dots < \gamma_0 \leq \frac{1}{2}$
 be the eigenvalues of $A$ in $[ - \frac{1}{2}, \frac{1}{2}].$
 
We then define the limit $d_{\gamma}$ as follows: 
\begin{equation*}
  d_{\gamma} := 
   \begin{cases}
     \dlim_{i \to \infty}  d_{\gamma,i}  &  (\gamma \neq \frac{1}{2}), \\ 
     \dlim_{i \to \infty} \sqrt{ |\log \eps_i| } \, d_{\frac{1}{2},i}
             &  (\gamma = \frac{1}{2})
   \end{cases} 
\end{equation*}
 and put  $E_\gamma=\Ker(A-\gamma).$

 Indeed, we can, step by step, decompose $d_{\gamma,i}$ on a part in 
 $\Ker (B \circ U^\ast)$ and a part which appears on a smaller behavior 
 in $\eps_i.$

 \noindent
 $\bullet$ first step: in $E_\frac{1}{2}.$ 
  Multiplying \eref{e18} by $\sqrt{\eps_i},$ we obtain that
   $\|B \circ U^{\ast} (d_{\frac{1}{2},i}) \|_{H^{- \frac{1}{2}} (\Sigma)}
      = O(\eps_i^{\frac{1}{2} -\gamma_1}).$
 We decompose $d_{\frac{1}{2},i}=
 \frac{1}{\sqrt{|\log\eps_i|}} d_{\frac{1}{2},i}^{(0)} + d_{\frac{1}{2},i}^{\perp}$
 along $\Ker(B \circ U^{\ast}  \rest_{E_{\frac{1}{2}}})$ and its orthogonal in 
 $E_\frac{1}{2}.$
  Then, 
 $ \| B\circ U^{\ast} (d_{\frac{1}{2},i}) \|_{H^{- \frac{1}{2}}(\Sigma)} 
  = O(\eps_i^{\frac{1}{2} -\gamma_1}) $
  implies 
 $ \| d_{\frac{1}{2},i}^{\perp} \|_{H^{\frac{1}{2}} (\Sigma)} 
   = O( {\eps_i}^{\frac{1}{2} - \gamma_1}).
 $
 So, 
 $$ d_{\frac{1}{2}} = \dlim_{i \to \infty} \sqrt{ |\log \eps_i| } \, 
  d_{\frac{1}{2},i} = \dlim_{i \to \infty} d_{\frac{1}{2},i}^{(0)} 
  \in \Ker(B \circ U^{\ast})
$$
 and if we write
   $d_{\frac{1}{2},i}^{\perp} = {\eps_i}^{\frac{1}{2} - \gamma_1} d_{i}^{(1)}$ and 
 re-introduce this in \eref{e18}, then it has the new expression 
 $(\ref{e18}^{\prime})$
\begin{multline*}
 \Big\| \eps_i^{\frac{1}{2}} \log (\eps_i) B \circ U^{\ast} 
  (d_{- \frac{1}{2},i}) + \sum_{j=2}^p \eps_i^{-\gamma_j} 
  B \circ U^{\ast} (d_{\gamma_j,i}) +\\ 
  {\eps_i}^{- \gamma_1} B \circ U^\ast (d_{i}^{(1)}+d_{\gamma_1,i})  
  \Big\|_{H^{- \frac{1}{2}} (\Sigma)}
    =  O( \sqrt{\eps_i}).
\end{multline*}
 \noindent
 $\bullet$ second step: in $E_\frac{1}{2}\oplus E_{\gamma_1}.$ 
 Multiplying by $\eps_i^{\gamma_1}$ in $(\ref{e18}^{\prime}),$ we obtain that
\begin{equation}\label{e19}
 \| B \circ U^\ast (d_{i}^{(1)}+d_{\gamma_1,i} )
  \|_{H^{- \frac{1}{2}} (\Sigma)} =O(\eps_i^{\gamma_1 -\gamma_2}).
\end{equation}
 We decompose $d_{i}^{(1)} + d_{\gamma_1,i} =
 d_{\gamma_1,i}^{(0)}+d_{\gamma_1,i}^{\perp}$
 along $\Ker(B\circ U^{\ast} \rest_{E_{\frac{1}{2}} \oplus E_{\gamma_1}})$
 and its orthogonal in 
 $E_{\frac{1}{2}} \oplus E_{\gamma_1}.$

 Now, the equation \eref{e19} says that 
 $ \| d_{\gamma_1,i}^{\perp} \|_{H^{\frac{1}{2}} (\Sigma)} 
     = O(\eps_i^{\gamma_1-\gamma_2}),$
 so
 $d_{\gamma_1} = \dlim_{i \to \infty} d_{\gamma_1,i}
  = \dlim_{i \to \infty} \Pi_{\{\gamma_1\}} (d_{\gamma_1,i}^{(0)})
 $
 and, as 
 $d_{\gamma_1,i}^{(0)} \in \Ker (B \circ U^{\ast} 
  \rest_{E_{\frac{1}{2}} \oplus E_{\gamma_1}} )$,
 extracting from $\Pi_{\{ \frac{1}{2} \}} (d_{\gamma_1,i}^{(0)})$ 
 a convergent subsequence, we can say that there exists 
 an $e_{\frac{1}{2}} \in E_{\frac{1}{2}}$ such that
 $$  d_{\gamma_1} + e_{\frac{1}{2}} \in \Ker (B \circ U^{\ast}).
$$
 On the other hand, if we can write 
$$ d_{\gamma_1,i}^{\perp} = {\eps_i}^{\gamma_1 - \gamma_2} d_{i}^{(2)},
$$
 then the new expression of \eref{e18} is
\begin{multline*}
 \Big\| \eps_i^{\frac{1}{2}} \log (\eps_i) B \circ U^{\ast} 
  (d_{- \frac{1}{2},i}) + \sum_{j=3}^p \eps_i^{-\gamma_j} 
  B \circ U^{\ast} (d_{\gamma_j,i}) + \\
    \eps_i^{- \gamma_2} B \circ U^\ast (d_{i}^{(2)}+d_{\gamma_2,i}) 
   \Big\|_{H^{- \frac{1}{2}} (\Sigma)}
    =  O( \sqrt{\eps_i}).
\end{multline*}
 $\bullet$ We can continue in this way until the term concerning $\gamma_p.$
It constructs terms 
\begin{align*}
  d_{\gamma_k,i}^{(0)} & \in \Big(E_{\frac{1}{2}} \oplus \dots \oplus 
    E_{\gamma_k}\Big) \cap \Ker (B \circ U^{\ast}), \\
  d_{i}^{(k+1)} & \in E_{\frac{1}{2}} \oplus \dots \oplus E_{\gamma_{k}}
\end{align*}
 with $0 \leq k \leq p.$
 If we decompose
   $d_{\gamma_k,i}^{(0)}= \dsum_{j=0}^k d_{\gamma_j,i}^{\gamma_k(0)}$ and 
   $d_{i}^{(k+1)} = \dsum_{j=0}^k d_{\gamma_j,i}^{(k+1)},$  then
\begin{align*}
 d_{\frac{1}{2},i} &= \frac{1}{\sqrt{|\log\eps_i|}} d_{\frac{1}{2},i}^{(0)}+
  {\eps_i}^{\frac{1}{2}- \gamma_1} d_{\frac{1}{2},i}^{\gamma_1(0)} + 
  {\eps_i}^{\frac{1}{2}- \gamma_2} d_{\frac{1}{2},i}^{\gamma_2(0)} + \cdots + 
  {\eps_i} \log (\eps_i) \, d_{\frac{1}{2},i}^{(p+1)}, \\
 d_{\gamma_1,i} &= \Pi_{\{\gamma_1\}} \big( d_{\gamma_1,i}^{(0)} \big) +
  {\eps_i}^{\gamma_1 -\gamma_2} d_{\gamma_1,i}^{\gamma_2(0)} +
  {\eps_i}^{\gamma_1 -\gamma_3} d_{\gamma_1,i}^{\gamma_3(0)} + \cdots.
\end{align*}

 Now, because all the families involved here (in finite numbers) are 
 bounded in a finite dimensional space, we can suppose, by successive 
 extractions, that they converge.
 We have
 $$ d_{\gamma} = \lim_{\eps_i\to 0} \Pi_{\{\gamma\}}
      \big( d_{\gamma,i}^{(0)} \big).
 $$
 It means that there exist elements 
  $\overline{\sigma}_{\gamma} = d_{\gamma} \in \Ker( A- \gamma),
   \, |\gamma| \leq \frac{1}{2}$ 
  such that there exists an $\eta_\gamma \in \Im (\Pi_{>\gamma})$ with
\begin{equation*}
 ( T \circ U^\ast) (\overline{\sigma}_\gamma + \eta_\gamma ) = 0,
\end{equation*}
 and if we denote
\begin{equation*}
\begin{split}
 \Pi_{(\gamma, \frac{1}{2}]} (\eta_\gamma) 
    = \sum_{\mu> \gamma} \eta^{\mu}_{\gamma},
\end{split}
\end{equation*}
 then we obtain
\begin{align*} 
  \Pi_{(- \frac{1}{2}, \frac{1}{2}]} \circ U (\phi_{1,i}(r)) \sim  
 & \sum_{- \frac{1}{2} \leq \mu< \gamma < \frac{1}{2} } r^{-\gamma} 
 ( \overline{\sigma}_{\gamma} + {\eps}_i^{\gamma-\mu} \eta^\gamma_\mu ) \\ 
 & \quad + r^{- \frac{1}{2}} \Big\{ |\log \eps_i|^{- \frac{1}{2}} 
  \overline{\sigma}_{\frac{1}{2}}+ 
  \sum_{- \frac{1}{2} \leq \mu < \frac{1}{2}}
    \eps_i^{ \frac{1}{2}- \mu} \eta^{\frac{1}{2}}_{\mu} \Big\}.
\end{align*}
 Here, the term $\eps^{- \mu}_i$ has to be replaced by 
  $\eps_i^{ \frac{1}{2}} \log \eps_i$ in the case of 
  $\mu = - \frac{1}{2}.$ 
\end{proof}

 \subsubsection{Conclusions on the side of $M_1$.}

 We now decompose $\phi_{1,i} = \phi_{1, \eps_i}$ near the singularity 
 as follows: Let 
\begin{equation*}
\begin{split}
 {\xi_1} \phi_{1,\eps_i} &= {\xi_1} \left\{ 
    \phi^{\leq - \frac{1}{2}}_{1,i} + 
    \phi^{ (- \frac{1}{2}, \frac{1}{2}] }_{1,i} +
    \phi^{> \frac{1}{2}}_{1,i} \right\}
\end{split}
\end{equation*}
 according to the decomposition, on the cone, of $\sigma_1$ along 
 the eigenvalues of $A$ respectively less than $- \frac{1}{2}$,
 in $(- \frac{1}{2}, \frac{1}{2}]$ and greater than $\frac{1}{2}.$ 

 We first remark that the expression and the convergence of 
 $\phi^{ (- \frac{1}{2}, \frac{1}{2}] }_{1,i}$ are given 
  by the preceding Proposition \ref{phi1W}.

 Now $\phi^{> \frac{1}{2}}_{1,i}$ and 
 $\widetilde{\psi}_{1,i} = \xi_1 P_{\eps_i} 
    (\Pi_{>\frac{1}{2}} \circ U (\phi_{2,i}(1))$ 
 have the same boundary value. 
 But, by Propositions \ref{inkerD2} and \ref{phi2}, we have 
\begin{equation*}
 \lim_{i \to \infty} U (\phi_{2, i}(1)) = U (\phi_{2}(1))
  \in \Im (\Pi_{>\frac{1}{2}}) 
 \hbox{ for the norm of } H^{\frac{1}{2}}(\Sigma).
\end{equation*}
 So, $\phi^{> \frac{1}{2}}_{1,i} - \widetilde{\psi}_{1,i}$ 
 can be considered in $H^1(M_1(0))$ by a prolongation by $0$ and
\begin{pro}\label{tild1}
 By uniform continuity of $P_{\eps_i},$ and the convergence property 
 just recalled
\begin{equation*}
 \lim_{i \to \infty} \| \widetilde{\psi}_{1,i} - {\xi_1} 
   P_{\eps_i} ( U ({\phi_2} \rest_{\Sigma} ) ) \|_{L^2 (M_1(\eps_i))} =0.
\end{equation*}
 On the other hand, ${\xi_1} P_{\eps_i} ( U({\phi_2} \rest_{\Sigma}) )$ 
 converges weakly to $0$ on the open manifold $M_1(0),$ 
 more precisely, for any fixed $\eta$ with $0 < \eta <1$ 
\begin{equation*}
 \lim_{i \to \infty} \| {\xi_1} P_{\eps_i} 
    ( U ( {\phi_2} \rest_{\Sigma} ) )\|_{L^2 (M_1(\eta))} =0.
\end{equation*}
\end{pro}


We remark finally that the boundary value of 
  $\phi^{\leq - \frac{1}{2}}_{1, i}$ is small. 
 We introduce for this term the cut-off function taken in \cite{ACP}:
\begin{equation*}
 {\xi}_{\eps_i}(r)=
   \begin{cases}
     \qquad  1   & \text{if } \ 2 \sqrt{\eps_i} \le r, \\
   \dfrac{1}{\log \sqrt{\eps_i} } 
     \log \left( \dfrac{ 2 \eps_i }{r} \right)
            & \text{if } \ 2 {\eps_i} \le r \le 2 \sqrt{\eps_i}, \\
       \qquad  0   & \text{if } \ r \le 2{\eps_i}.
  \end{cases}
\end{equation*}
\begin{pro}
 $ \dlim_{i \to \infty} \left\| (1- {\xi}_{\eps_i} ) \, {\xi_1} 
   \phi^{\leq - \frac{1}{2}}_{1, i} \right\|_{L^2(M_1(\eps_i))} =0.$
\end{pro}
 This is a consequence of the estimates of Lemmas \ref{l1} and \ref{l2}:
 we remark that by the same argument, we obtain also
 $ \| {\xi_1} \phi^{\leq - \frac{1}{2}}_{1, i}(r) \|_{L^2(\Sigma)}
    \leq C \sqrt{r} $ so 
 $$ \left\| (1- {\xi}_{\eps_i}) {\xi_1} \phi^{\leq - \frac{1}{2}}_{1,i}
     \right\|_{L^2(M_1 (\eps_i))} = O ({\eps_i}^{\frac{1}{4}}).$$

\begin{pro}\label{psi1} The forms 
\begin{equation*}
\begin{split}
  \psi_{1, i} &= (1 - \xi_1) \phi_{1, i} + \left( 
  {\xi_1} \phi^{> \frac{1}{2}}_{1, i} - \widetilde{\psi}_{1, i} \right) 
   + \xi_{\eps_i} {\xi_1} \phi^{\leq - \frac{1}{2}}_{1, i} 
   + {\xi_1} U^{\ast} (\overline{\sigma}_{0, i}^{\frac{1}{2}} )
\end{split}
\end{equation*}
  belong to $\Dom (D_{1,\min})$ and define a bounded family. 
\end{pro}
\begin{proof}
 We will show that each term is bounded. 
 For the last one, it is a consequence of Proposition \ref{phi1W}.
 For the first one, it is already done in Lemma \ref{l0}.  
 For the second one, we remark that 
\begin{equation}\label{feps}
\begin{split}
 f_i &:= ( \partial_r + \frac{A}{r} ) U \left( {\xi_1} 
    \phi^{> \frac{1}{2}}_{1,i} - \wt{\psi}_{1, i} \right) \\
   &= {\xi_1} \big( \partial_r + \frac{A}{r} \big) 
     (U \phi^{> \frac{1}{2}}_{1,i} )+ \partial_r ( {\xi_1} )
      U \left( \phi^{> \frac{1}{2}}_{1,i}  - P_{\eps_i} 
      ( \Pi_{> \frac{1}{2}} \phi_{2,i}(1) ) \right)
\end{split}
\end{equation}
 is uniformly bounded in $L^2( \overline{M}_1),$ because of (\ref{contL2}).
 This estimate (\ref{contL2}) shows also that the $L^2$-norm of 
 $ \phi^{> \frac{1}{2}}_{1,i} - \tilde{\psi}_{1,i}$ is bounded.
 
 For the third one, we use the estimate due to the expression of the quadratic 
 form. The estimate that 
  $\int_{\mcC_{r,1}} | D_1({\xi_1} \phi^{\leq - \frac{1}{2}}) |^2 d \mu
   \le \Lambda$ gives that
\begin{equation}\label{val-}
  \Big\| \sigma^{\leq - \frac{1}{2}}_1(r) \Big\|^2_{L^2(\Sigma)} 
       \leq \Lambda r |\log r|
\end{equation}
 by the same argument as in Lemmas \ref{l1} and \ref{l2}. Now 
\begin{equation*}
\begin{split}
 \left\| D_1 (\xi_{\eps_i} {\xi_1} \phi^{\leq - \frac{1}{2}}_{1,i} )
     \right\|_{L^2( \overline{M}_1 )}
  &\leq  \left\| \xi_{\eps_i} D_1 ({\xi_1} \phi^{\leq - \frac{1}{2}}_{1,i}) 
     \right\|_{L^2( \overline{M}_1)}   +  \left\| |d \xi_{\eps_i} | \cdot 
      {\xi_1} \phi^{\leq - \frac{1}{2}}_{1,i} 
     \right\|_{L^2 ( \overline{M}_1)}  \\
  &\leq  \left\| D_1({\xi_1} \phi^{\leq - \frac{1}{2}}_{1,i} ) 
      \right\|_{L^2 (\mcC_{\eps_i,1}) } + \left\| |d \xi_{\eps_i}| \cdot 
        {\xi_1} \phi^{\leq - \frac{1}{2}}_{1, i}  
        \right\|_{L^2(\mcC_{\eps_i, \sqrt{\eps_i}})}.
\end{split}
\end{equation*}
 The first term is bounded and, with $|A| \geq \frac{1}{2}$ for this term, 
 and the estimate (\ref{val-}), we have 
\begin{align*}
 \left\| |d \xi_{\eps_i}| {\xi_1} \phi^{\leq - \frac{1}{2}}_{1, i} 
   \right\|^2_{L^2 (\mcC_{\eps_i, \sqrt{\eps_i}} )} 
   &\leq \frac{4 \Lambda}{| \log {\eps_i} |^2} 
           \int_{\eps_i}^{\sqrt{\eps_i}} \frac{\log r }{r} \, dr  
   \leq  \frac{3}{2} \Lambda.
\end{align*}
 This completes the proof.
\end{proof}

 In fact, the decomposition used here is almost orthogonal:
\begin{lem} \label{lem:almost-orthogonal}
  There exists $\beta>0$ such that
$$
 \big( \phi^{> \frac{1}{2}}_{1,i} - \widetilde{\psi}_{1,i}, \, 
    \widetilde{\psi}_{1,i} \big)_{L^2(M_1(\eps_i))} = O({\eps_i}^\beta).
$$
\end{lem}
\subsubsection*{Proof of Lemma \ref{lem:almost-orthogonal}}\tir 
 If we decompose the terms under the eigenspaces of $A$, we see that only
 the eigenvalues in $(\frac{1}{2}, \infty )$ are involved.
 With  $f_{i} = \sum_{\gamma > \frac{1}{2}} f^{\gamma}$ and 
 $ U( \phi^{> \frac{1}{2}}_{1, i} - \wt{\psi}_{1,i} )
      = \sum_{\gamma > \frac{1}{2}} \phi_0^{\gamma},$ 
 the equation (\ref{feps}) and the fact that 
  $(\phi^{> \frac{1}{2}}_{1,i} - \wt{\psi}_{1,i}) (\eps_i)=0$ 
  imply 
$$
 \phi_0^{\gamma}(r) = r^{-\gamma} \int_{\eps_i}^r 
       \rho^{\gamma} f^{\gamma}(\rho) \, d \rho.
$$
 Then for each eigenvalue $\gamma > \frac{1}{2}$ of $A$
\begin{align*}
 (\phi_0^{\gamma}, \, \tilde{\psi}_{1,i}^{\gamma} 
    )_{L^2 (\mcC_{\eps_i,1})}
  &= {\eps_i}^{\gamma - \frac{1}{2}} \int_{\eps_i}^1 r^{-2 \gamma}
       \int_{\eps_i}^r \rho^{\gamma} 
         (\sigma_{\gamma}, f^{\gamma}(\rho) )_{L^2(\Sigma)} \, d \rho \\
  &=  {\eps_i}^{\gamma - \frac{1}{2}} \int_{\eps_i}^1 
        \frac{r^{-2 \gamma +1} }{2 \gamma-1} \cdot r^{\gamma} \cdot 
       (\sigma_{\gamma}, \, f^\gamma(r) )_{L^2(\Sigma)} \, dr  \\ 
  &\hspace{1.5cm} + \frac{ {\eps_i}^{\gamma - \frac{1}{2}} }{2\gamma -1}
    \int_{\eps_i}^1 \rho^{\gamma} 
      (\sigma_{\gamma}, \, f^{\gamma}(\rho) )_{L^2(\Sigma)} \, d \rho.
\end{align*}
 Thus, if $\gamma > \frac{3}{2},$ we have the upper bound 
\begin{align*}
  | (\phi_0^{\gamma}, \, \wt{\psi}_{1,i}^{\gamma} 
     )_{ L^2 (\mcC_{{\eps_i},1}) } |
  &\leq  {\eps_i}^{\gamma - \frac{1}{2}} \int_{\eps_i}^1 
     \frac{r^{- \gamma+1}}{2\gamma -1} 
      | (\sigma_{\gamma}, \, f^{\gamma}(r) )_{L^2(\Sigma)} | \, dr \\
  &\hspace{1cm} + \frac{{\eps_i}^{\gamma- \frac{1}{2}} }{
      (2\gamma-1) \sqrt{2 \gamma +1} } \| \sigma_{\gamma} \|_{L^2 (\Sigma)}
       \cdot \| f^{\gamma} \|_{L^2(\mcC_{{\eps_i},1})} \\
  &\leq C {\eps_i}^{\gamma - \frac{1}{2}} \| \sigma_{\gamma} \|_{L^2 (\Sigma)}
      \frac{ {\eps_i}^{ \frac{-2 \gamma+3}{2} } }{ (2 \gamma -1)
      \sqrt{2\gamma-3} } \| f^{\gamma} \|_{L^2(\mcC_{{\eps_i},1}) } \\
 &\hspace{1cm} + \frac{{ \eps_i}^{\gamma - \frac{1}{2}} 
    }{ (2 \gamma -1)\sqrt{2 \gamma +1}} \| \sigma_{\gamma} \|_{L^2 (\Sigma)} 
    \cdot  \| f^{\gamma} \|_{L^2(\mcC_{\eps_i,1})},
\end{align*}
 while, for $\gamma = \frac{3}{2}$ the first term is 
  $O ( {\eps}_i \sqrt{ |\log \eps_i| } )$ and 
  for $\frac{1}{2} < \gamma < \frac{3}{2},$ 
  it is $O( {\eps_i}^{\gamma - \frac{1}{2}} ).$
 In short, we have 
\begin{equation*} 
 | (\phi_0^{\gamma}, \, \widetilde{\psi}_{1, i}^{\gamma} 
   )_{L^2(\mcC_{\eps_i,1})} |
   \leq C {\eps_i}^{\beta} \| \sigma_{\gamma} \|_{L^2 (\Sigma)} 
      \cdot \| f^{\gamma} \|_{L^2( \mcC_{\eps_i,1} ) },
\end{equation*}
 if $\beta>0$ satisfies $\gamma \geq \beta + \frac{1}{2}$ for all eigenvalues 
 $\gamma$ of $A$ in $(\frac{1}{2}, \infty).$
 This estimate gives Lemma \ref{lem:almost-orthogonal}. 
 \hspace{12.5cm}  $\square$
\begin{cor}\label{conc}
 There exists from 
 $\{ \psi_{1, i} + \phi^{(-\frac{1}{2}, \frac{1}{2})}_{1,i} \}_i$ 
 a subfamily which converges in $L^2$ to a form $\phi_1 \in \Dom(D_{1,W})$ 
 which satisfies on the open manifold $M_1(0)$ the equation
    $ \Delta \phi_1 = \lambda \phi_1.$
  Moreover, 
\begin{equation}\label{decompose}
  \| \phi_1 \|^2_{L^2(M_1(0)) } + \| \wt{\phi_2} \|^2_{L^2(\wt{M}_2)}  + 
         \| \overline{\sigma}_{\frac{1}{2}} \|^2_{L^2(\Sigma)} = 1,
\end{equation}
 where $\wt{\phi_2}$ is the prolongation of ${\phi_2}$ by 
  $P_2( {\phi_2} \rest_{\Sigma} )$ on $\wt{M_2}$, and 
 $  \overline{\sigma}_{\frac{1}{2}} $ given by Proposition $\ref{phi1W}$.
\end{cor}
\begin{proof}
  Indeed, the family 
   $\{ \psi_{1,i} + \phi^{(-\frac{1}{2}, \frac{1}{2})}_{1,i} \}_i$ 
 is bounded in $\Dom (D_{1, \max})$, one can then extract a subfamily 
 which converges in $L^2 (\overline{M}_1, \overline{g}_1)$.
 But we know that $\widetilde{\psi}_{1,i}$ converges to $0$
  in any $M_1(\eta)$, the conclusion follows. 
  We obtain also, with the help of Lemma \ref{lem:almost-orthogonal} that 
\begin{equation*}
\begin{split}
  1- & \big\{  \| \phi_1 \|^2_{L^2( M_1(0)) } + 
     \|\phi_2 \|^2_{L^2 (M_2)}  \big\}  \\
  &= \dlim_{i \to \infty} 
    \big\{ \| \widetilde{\psi}_{1, i} \|^2_{L^2(M_1(\eps_i))} + 
       \Big\| {\xi_1} U^\ast( \frac{1}{\sqrt{ |\log \eps_i|} } 
       r^{- \frac{1}{2}} \overline{\sigma}_{\frac{1}{2}} ) 
       \Big\|^2_{L^2(M_1(\eps_i))} \big\}.
\end{split}
\end{equation*}
   We remark that, by Proposition \ref{tild1}, $\phi_2 =0$ implies 
 $ \dlim_{i \to \infty} \| \widetilde{\psi}_{1,i} 
   \|_{L^2(M_1(\eps_i))} =0. $
 In fact, one has by (\ref{contL2})
\begin{equation}\label{p2phi2}
  \lim_{i \to \infty} \| \widetilde{\psi}_{1,i} \|_{L^2(M_1(\eps_i))}
       = \| P_2(U{\phi_2} \rest_{\Sigma} ) \|_{L^2( \wt{M}_2) }.
\end{equation}

 Finally, one has
\begin{equation}
  \lim_{i \to \infty} \Big\| {\xi_1} U^{\ast} 
    ( \frac{1}{\sqrt{|\log \eps_i|}} r^{- \frac{1}{2}} 
     \overline{\sigma}_{\frac{1}{2}} )  \Big\|_{L^2 (M_1(\eps_i) )}
     = \| \overline{\sigma}_{\frac{1}{2}} \|_{L^2(\Sigma)}.
\end{equation}
\end{proof}
\subsection{Lower bound, the end.}

 Let us now $\{ \phi_1(\eps), \dots, \phi_N(\eps) \}$ be 
 an orthonormal family of eigenforms of the Hodge-de Rham operator, 
 associated with the eigenvalues 
 $\lambda_1(\eps), \dots, \lambda_N(\eps)$. 
 We can make the same procedure of extraction for all the families. 
 This gives, in the limit domain, a family
 $\{ (\phi^j_{1}, {\phi^j_{2}}, \overline{\sigma}^j_{\frac{1}{2}} 
 ) \}_{1 \leq j \leq N}.$ 
 We already know by \cref{conc} that each element has norm $1$. 
 If we show that they are orthogonal, then we are done, 
 by applying the min-max formula to the limit problem \eref{limitpb}.
\begin{lem}
  The limit family is orthonormal in $\mcH_{\infty}.$
\end{lem}
\begin{proof}
 If we follow the procedure for one index, up to terms converging
  to zero, we had decomposed the eigenforms $\phi_j(\eps)$ 
  on $M_{\eps}$ into three terms:
\begin{equation*}
\begin{split}
  \Phi^j_{\eps} &= \psi_{1, i} + \phi^{(-\frac{1}{2}, 
       \frac{1}{2})}_{1, i}, \\
  \widetilde{\Phi}^j_{\eps} &= \widetilde{\psi}_{1, i}, \\
  \overline{\Phi}^j_{\eps} &= U^{\ast} \Big( \frac{1}{ \sqrt{|\log \eps|} }
      r^{- \frac{1}{2}} \overline{\sigma}^j_{\frac{1}{2}} \Big).
\end{split}
\end{equation*}
 Let $a \neq b$ be two indices.
 If we apply \lref{lem:almost-orthogonal} to any linear combination of 
 $\phi_a (\eps)$ and  $\phi_b (\eps),$ we obtain that
\begin{equation*}
\begin{split}
 \lim_{i \to \infty} \left\{ ( \Phi^a_{\eps_i}, \, 
       \widetilde{\Phi^b}_{\eps_i} )_{L^2(M_1(\eps_i))} + 
  ( \Phi^b_{\eps_i}, \, \widetilde{\Phi^a}_{\eps_i} 
     )_{L^2(M_1(\eps_i))} \right\}   &=0.
\end{split}
\end{equation*}
  If we apply \eref{p2phi2}, we obtain
\begin{equation*}
\begin{split}
  \lim_{i \to \infty} \left\{ ( \widetilde{\Phi^a}_{\eps_i}, 
   \widetilde{\Phi^b}_{\eps_i} )_{L^2(M_1(\eps_i))} +
   ( \phi^a_{2,\eps}, \phi^b_{2,\eps} )_{L^2(M_2)} \right\} &=
  ( \widetilde{\phi}^a_{2}, \widetilde{\phi}^b_{2} )_{L^2(\wt{M_2})}.
\end{split}
\end{equation*}
 Then finally, from $( \phi_a(\eps), \phi_b(\eps) )_{L^2 (M_{\eps})} =0$, 
 we conclude that 
\begin{equation*}
  (\phi^a_{1}, \phi^b_{1} )_{L^2(\barM)} + 
  (\phi^a_{2}, \phi^b_{2} )_{L^2(\wt{M_2})} + 
  (\overline{\sigma}^a_{\frac{1}{2}}, 
       \overline{\sigma}^b_{\frac{1}{2}} )_{L^2(\Sigma)} = 0.
\end{equation*}
\end{proof}

\begin{pro}\label{limit}
 The multiplicity of $0$ in the limit spectrum is given by the sum 
  $$ \dim \Ker (\Delta_{1,W}) + \dim \Ker (\mcD_2) + i_{\frac{1}{2}}, $$
 where $i_{\frac{1}{2}}$ denotes the dimension of the vector space 
 $\mcI_{\frac{1}{2}},$  see $\eqref{I1/2}$, of extended solutions 
 $\omega$ on $\widetilde{M_2}$ introduced by Carron \cite{C}, 
 corresponding to a boundary term on restriction to $r=1$ 
 with non-trivial component in $\Ker(A- \frac{1}{2}).$

 If the limit value $\lambda \neq 0,$ then it belongs to the positive 
 spectrum of the Hodge-de Rham operator $\Delta_{1,W}$ on $\overline{M_1},$ 
 with the space $W$ defined in $\eqref{defW}$.
\end{pro}
\begin{proof}
 The last process, with in particular (\ref{p2phi2}) and (\ref{1/2}), 
 constructs in fact an element in the limit Hilbert space
$$
 \mcH_{\infty} := L^2 (\overline{M_1}) \oplus \Ker (\widetilde{D_2}) 
       \oplus  \mcI_{\frac{1}{2}}.
$$
 This process is clearly {\em isometric} in the sense that if we have 
 an orthonormal family $\{ \phi_j (\eps_i) \}_j \ (1 \leq  j \leq N),$ 
 we obtain at the limit an orthonormal family,
 where $\mcH_{\infty}$ is defined as an orthogonal sum of 
  the Hilbert spaces.
 And if we begin with eigenforms of $\Delta_{\eps_i},$  we obtain 
 at the limit  eigenforms of $\Delta_{1,W} \oplus \{ 0 \} \oplus \{ 0 \}.$
 The last calculus implies that 
  $ \liminf_{i \to \infty} \lambda_N (\eps_i) \geq \lambda_N.$
\end{proof}
\begin{rem}
 In order to understand this result, it is important to remember 
  when the eigenvalue $\frac{1}{2}$ occurs in the spectrum of $A$.
 By the expression $\eref{vpA}$, we find that it occurs exactly
\begin{itemize}
 \item for n even, if $\frac{3}{4}$ is an eigenvalue of the Hodge-de Rham
  operator $\Delta_\Sigma$ acting on coexact forms of degree 
  $\frac{n}{2}$ or $\frac{n}{2}-1$ of the submanifold $\Sigma.$
 \item  for n odd, if $0$ is an eigenvalue of $\Delta_\Sigma$ on forms
  of degree $\frac{n-1}{2}, \, \frac{n+1}{2}$,  but also 
  if $1$ is an eigenvalue of coexact forms of degree $\frac{n-1}{2}$ 
  on $\Sigma.$
\end{itemize}
  A dilation of the metric on $\Sigma$ permits to avoid positive 
 eigenvalues, but harmonic forms of degree 
  $\frac{n-1}{2}$ or $\frac{n+1}{2}$ on $\Sigma$ can not be avoid.

 Moreover, Carron has proved (Theorem $0.6$ in \cite{C2}) that 
 the extended index depends only on geometry at infinity: 
 these harmonic forms on $\Sigma$ will indeed create half-bound states,
  and then, small eigenvalues will always appear.
\end{rem}

 \section{Harmonic forms and small eigenvalues.}

 It would be interesting to know how many small (but non-zero) 
 eigenvalues appear. 
 For this purpose, we can use the topological meaning of harmonic forms.

\subsection{Cohomology groups.}
 The topology of $M_{\eps}$ is independent of $\eps \neq 0$ and 
 can be apprehended by the Mayer-Vietoris exact sequence:
\begin{equation*}
 \cdots \to H^{p}(M_\eps) \stackrel{\text{res}}{\longrightarrow} 
  H^{p}(M_1(\eps)) \oplus H^{p}(M_2) \stackrel{\text{dif}}{\longrightarrow} 
  H^{p}(\Sigma) \stackrel{\text{ext}}{\longrightarrow} H^{p+1}(M_\eps) 
  \to \cdots.
\end{equation*}

 As already mentioned, the space $\Ker(\mcD_2) \oplus \mcI_{\frac{1}{2}}$ 
 can be sent in $H^{\ast} (M_2)$. 
 More precisely, Hausel, Hunsicker and Mazzeo in \cite{HHM}, Theorem 1.A, 
 p.490, have proved that the space of the $L^2$-harmonic forms 
 ${\mcH}^k_{L^2}(\widetilde{M_2})$ on $\widetilde{M_2}$ is given by:
\begin{equation} \label{L2-cohomology}
\begin{split}
  {\mcH}^k_{L^2}(\widetilde{M_2}) &\cong
 \begin{cases}
   H^k(M_2, \Sigma)  &  \text{if }  k < \frac{n+1}{2}, \\
   \Im \Big( H^{\frac{n+1}{2}}(M_2,\Sigma) 
     \to H^{\frac{n+1}{2}}(M_2) \Big) & \text{if } k= \frac{n+1}{2}, \\
   H^k(M_2)   & \text{if }  k > \frac{n+1}{2}.
 \end{cases}
\end{split}
\end{equation}
 We note that the space of $L^2$-harmonic forms is equal to that of
 $L^2$-harmonic fields, or the Hodge cohomology group,
 since $\widetilde{M_2}$ is complete.

 For $\overline{M}_1$, we can use the results of Cheeger. 
 Following \cite{Ch80} and \cite{Ch83}, we know that 
 the intersection cohomology groups $ I\!H^* (\overline{M}_1)$  
 of $\overline{M}_1$ coincide with 
  $\Ker(D_{1,\max} \circ D_{1,\min})$, if $H^{\frac{n}{2}}(\Sigma)=0.$
 And we know also that 
\begin{equation}
   I\!H^p (\overline{M}_1) \cong
  \begin{cases}
     H^p (M_1 (\eps))   & \text{ if } p \leq  \frac{n}{2},  \\
     H^p_c(M_1 (\eps))  & \text{ if } p \geq  \frac{n}{2}+ 1.
  \end{cases}
\end{equation}
 These results can be used for our study only if $D_{1,\max}$ and 
 $D_{1,\min}$ coincide.  This occurs if and only if $A$ has no 
 eigenvalues in the interval $(- \frac{1}{2}, \frac{1}{2}).$ 
 As a consequence of the expression of the eigenvalues of $A$, 
 recalled in \eref{vpA}, this is the case if and only if
\begin{itemize}
  \item{\bf for n odd,} the operator $\Delta_\Sigma$ has no eigenvalues
   in $(0, 1)$ on coexact forms of degree $\frac{n-1}{2}$,
  \item {\bf for n even,} the operator $\Delta_\Sigma$ has no eigenvalues 
  in $(0, \frac{3}{4})$ on coexact forms of $\frac{n}{2}$ or $\frac{n}{2}-1,$
   and $H^{ \frac{n}{2} }(\Sigma)=0.$
\end{itemize}

{\it Thus, if $D_{1,\max}=D_{1,\min},$ which implies 
  $ H^{ \frac{n}{2}} (\Sigma)=0$ in the case where $n$ is even, 
  then the map
$$
  H^{ \frac{n}{2}}(M_\eps) \stackrel{\text{res}}{\longrightarrow} 
  H^{\frac{n}{2}}(M_1(\eps)) \oplus H^{\frac{n}{2}}(M_2)
$$
 is surjective and then any small eigenvalue in this degree must 
 come from an element of $\Ker (\mcD_2) \oplus \mcI_{\frac{1}{2}}$ 
 sent to $0$ in $H^{\frac{n}{2}}(M_2).$
 In this case also the map
$$
 H^{ \frac{n}{2}+1}(M_\eps) \stackrel{\text{res}}{\longrightarrow} 
  H^{\frac{n}{2}+1}(M_1(\eps)) \oplus  H^{\frac{n}{2}+1}(M_2)
$$
 is injective, so there may exist small eigenvalues in this degree.
}
\subsection{Some examples.}

  We exhibit a general procedure to construct new examples as follows:
  Let $W_i$, $i=1,2,$ be two compact Riemannian manifolds with boundary  
  $\Sigma_i$ and dimension $n_i +1$ such that $n_1 + n_2 = n \geq 2.$
 We can apply our result to $M_1 := W_1 \times \Sigma_2$ and 
  $M_2 := \Sigma_1 \times W_2$.
 The manifold $M_{\eps}$ is always diffeomorphic to $M = M_1 \cup M_2.$

 For instance, let $v_2$ be the volume form of $(\Sigma_2, h_2).$ 
 It defines a harmonic form on $M_1$ and this form will appear in
 the limit spectrum if, transplanted on $\overline{M_1}$, it defines 
 an element in the domain of the operator $\Delta_{1,W}.$

 In the writing introduced in Section \ref{GBconic}, this element 
 corresponds to $\beta =0$ and $\alpha = r^{ \frac{n}{2} - n_2} v_2$ and 
 the expression of $A$ gives that 
$$
 A (\beta,\alpha) = \left( n_2 - \frac{n}{2} \right) (\beta,\alpha).
$$
 If $\frac{n}{2} - n_2> 0$, then $(\beta, \alpha)$ is in the domain of
  $D_{1,\max} \circ D_{1,\min}$ and if $n_2 = \frac{n}{2}$, it is 
  in the domain of $\Delta_{1,W}$ for the eigenvalue $0$ of $A.$

 So, if we know that $H^{n_2}(M)=0$ or, more generally,  
  $\dim H^{n_2}(M) < \dim H^{n_2}(\Sigma_2)$ in the case where
 $\Sigma_2$ is not connected, then this element will create 
 a small eigenvalue on $M_\eps.$
 This is the case, if $D^k$ denotes the unit ball in $\R^k$, for
$$
  W_1=D^{n_1+1} \, \hbox{ and } \, W_2= D^{n_2+1} \, 
   \hbox{ for $n_2\leq n_1.$ }
$$
 Then,  $M= {\Sphere}^{n_1 + n_2 +1}$ and we obtain 
\begin{cor}
 For any degree $k$ and any $\eps>0$, there exists a metric on $\Sphere^m$ 
  such that the Hodge-de Rham operator acting on $k$-forms admits an 
  eigenvalue  smaller than $\eps.$
  We can see that, for $k < \frac{m}{2},$ it is in the spectrum of coexact 
  forms, and by the duality, for $k \geq \frac{m}{2}$ in the spectrum of 
  exact $k$-forms. 
\end{cor}
 Indeed, the case $k < \frac{m}{2}$ is a direct application, as explained 
  above. We see that our quasi-mode is coclosed.
 Thus, in the case where $m$ is even, if $\omega$ is an eigenform of degree 
  $\frac{m}{2} - 1$ with small eigenvalue,
 then $d \omega$ is a closed eigenform with the same eigenvalue and 
  degree $\frac{m}{2}$. 
  Finally, the case $k > \frac{m}{2}$ is obtained by the Hodge duality.
 We remark that, in the case $k=0$ we recover {\it Cheeger's dumbbell},
  and also that this result has been proved by Guerini in \cite{Gue} 
  with another deformation, although he did not give the convergence 
  of the spectrum.

\vspace{0.5cm}
 By the surgery of the previous case, we obtain, for
 $$ W_1 := \Sphere^{n_1} \times [0,1]  \, \hbox{ and } \,
    W_2 := D^{n_2+1} \, \hbox{for $0 \leq n_2 <n_1,$ and 
    $n= n_1 + n_2 \geq 2$}
$$
 that $\Sigma_1 = \Sphere^{n_1} \sqcup \Sphere^{n_1},$ 
  $ \Sigma_2 = \Sphere^{n_2}$ and  $M= \Sphere^{n_1} \times \Sphere^{n_2+1}.$
 The volume form $v_2 \in H^{n_2}(\Sigma_2)$ defines again a harmonic form 
 on $\overline{M}_1$ and, 
 since $H^{n_2}( \Sphere^{n_1} \times \Sphere^{n_2+1}) =0$, if $n_2 <n_1,$ 
 then $v_2$ defines a small eigenvalue on $n_2$-forms of $M_{\eps}.$

 Thus, by the duality, we obtain
\begin{cor}
 For any $k, l \ge 0$ with $0 \leq  k-1 < l$ and any $\eps >0$, 
  there exists a metric on $\Sphere^l \times \Sphere^{k}$ 
  such that the Hodge-de Rham operator acting on $(k-1)$-forms and 
  on $(l+1)$-forms admits an eigenvalue smaller than $\eps.$
\end{cor}
 This corollary is also a consequence of the previous one: 
  we know that there exists a metric on $\Sphere^{k}$ whose Hodge-de Rham 
  operator admits a small eigenvalue on $(k-1)$-forms, 
 and this property is maintained on $\Sphere^{l} \times \Sphere^{k+1}.$

 With the same construction, we can exchange the roles of $M_1$ and $M_2$: 
 the two volume forms of $\Sphere^{n_1} \sqcup \Sphere^{n_1}$ create 
 one $n_1$-form with small but non-zero eigenvalue on 
  $\Sphere^{n_1} \times \Sphere^{n_2+1}$, if $n_1 \leq n_2 +1.$ 
  By the duality, we obtain an $(n_2+1)$-form with small eigenvalue.
   So, with new notations, we have obtained
\begin{cor}
 For any $k<l$ with $k+l \geq 3$ and any $\eps>0$, there exists a metric on
 $\Sphere^l \times \Sphere^k$ such that the Hodge-de Rham operator acting 
 on $l$-forms and on $k$-forms admits a positive eigenvalue smaller 
 than $\eps.$
\end{cor}

 More generally, by repeating the $(k-1)$-dimensional surgery by $L$-times, 
 we obtain the following:
\begin{pro}[\cite{Sha-Yang[91]}] \label{prop:Sha-Yang[91]}
 The connected sum of the $L$-copies of the product spheres 
 $\displaystyle { \underset{i=1}{ \overset{L}{\sharp} } } 
   ( {\Sphere}^k \times {\Sphere}^l ) $
 can be decomposed as follows:
\begin{equation*}
\begin{split}
  { \underset{i=1}{ \overset{L}{\sharp} } } 
         ( {\Sphere}^k \times {\Sphere}^l )
    &\cong  \left( {\Sphere}^{k-1} \times  \Big( {\Sphere}^{l+1} \setminus 
        \coprod_{i=0}^{L} D^{l+1}_i \Big) \right) \bigcup_{\partial} 
        \left( D^{k} \times  \coprod_{i=0}^{L}  {\Sphere}^{l}_i \right).
\end{split}
\end{equation*}
\end{pro}
\begin{rem}
 J-P.\ Sha and D-G.\ Yang \cite{Sha-Yang[91]} constructed a Riemannian 
 metric of positive Ricci curvature on this manifold.
 More generally, see also Wraith \cite{Wraith[07]}.
\end{rem}

 As similar way using Proposition \ref{prop:Sha-Yang[91]},
 we can obtain the small positive eigenvalues on the connected sum of 
 the $L$-copies of the product spheres
 ${ \underset{i=1}{ \overset{L}{\sharp} } } 
       \Big( {\Sphere}^k \times {\Sphere}^l \Big). $

 All these examples use the spectrum of $\overline{M}_1.$ 
 We can obtain also examples using the reduced $L^2$-cohomology group 
 of $\widetilde{M}_2,$ which is given by \eqref{L2-cohomology}
  (Hausel, Hunsicker and Mazzeo \cite{HHM}).

 Suppose now that $n = \dim \Sigma$ is odd.
 Then, we have the long exact sequence
$$
  \cdots \to H^k (M_2, \Sigma) \to H^k(M_2)  \to H^k(\Sigma) 
   \to H^{k+1}(M_2, \Sigma) \to \cdots.
$$
 For $k= \frac{n-1}{2}$, the space $H^k (M_2, \Sigma)$ is isomorphic 
 to the reduced $L^2$-cohomology group of $\widetilde{M_2}.$
 If $H^{ \frac{n-1}{2}}(\Sigma)$ is non-trivial, 
 then any non-trivial harmonic $k$-form on $\Sigma$ will create 
 an extended solution,  corresponding to an eigenvector of $A$ 
 with eigenvalue $\frac{1}{2}.$

 For example, take $\Sigma = \Sphere^{k} \times \Sphere^{k+1}$ for 
 $k= \frac{n-1}{2}$, then $H^k(\Sigma)$ is non-trivial.
 Any non-trivial form $\omega \in H^k(\Sigma)$ sent to 
 $0 \in H^{k+1}(M_2, \Sigma)$ comes from an element
  $\widetilde{\omega} \in H^k(M_2)$ which is not in the reduced 
 $L^2$-cohomology group of $\widetilde{M}_2$ by \eqref{L2-cohomology}.


\end{document}